\documentclass[11pt]{amsart}

\usepackage{amssymb}
\usepackage{amsmath}
\usepackage{amsthm}
\usepackage{color}
\usepackage{euscript}
\usepackage[linktoc=page,colorlinks,citecolor = green!50!black,linkcolor= blue,urlcolor = blue]{hyperref}
\usepackage{enumitem}
\usepackage{tikz-cd} 
\usepackage{tikz}
\usepackage{mathrsfs}
\usepackage[all]{xy}
\usepackage{verbatim}
\usepackage{comment}
\usepackage{cleveref}
\usepackage[margin=1in]{geometry}

\theoremstyle{plain}
\newtheorem{thm}{Theorem}[section]
\newtheorem{lem}[thm]{Lemma}

\newtheorem{cor}[thm]{Corollary}
\newtheorem{prop}[thm]{Proposition}

\newtheorem{question}[thm]{Question}

\theoremstyle{definition}
\newtheorem{defn}[thm]{Definition}
\newtheorem{exmp}[thm]{Example}

\theoremstyle{definition}
\newtheorem{note}[thm]{Note}
\newtheorem{rk}[thm]{Remark}

\newcommand{\bC}{{\mathbb C}}

\newcommand{\bF}{{\mathbb F}}

\newcommand{\bN}{{\mathbb N}}

\newcommand{\bP}{{\mathbb P}}
\newcommand{\bQ}{{\mathbb Q}}
\newcommand{\bR}{{\mathbb R}}
\newcommand{\bS}{{\mathbb S}}

\newcommand{\bZ}{{\mathbb Z}}

\newcommand{\cA}{{\mathcal A}}

\newcommand{\cC}{{\mathcal C}}

\newcommand{\cL}{{\mathcal L}}
\newcommand{\cM}{{\mathcal M}}
\newcommand{\cN}{{\mathcal N}}

\newcommand{\cU}{{\mathcal U}}

\newcommand{\Q}{\mathbb{Q}}

\newcommand{\sheafhom}{\mathcal{H} \kern -.5pt \mathit{om}}
\newcommand{\derivedsheafhom}{\mathcal{R}\mathcal{H} \kern -.5pt \mathit{om}}

\newcommand{\colim}{\operatornamewithlimits{colim}\limits}

\newcommand{\tate}[1]{\widehat{#1}_{S^1}^*}
\newcommand{\Sp}{\operatorname{Sp}}
\newcommand{\Fun}{\operatorname{Fun}}

\newcommand{\Top}{\operatorname{Top}}

\numberwithin{equation}{section}
\numberwithin{figure}{section}
\title[Generalized homology from Floer homology]{Recovering generalized homology from Floer homology: the complex oriented case}

\author{Laurent C\^{o}t\'{e}}
\address{Universit\"at Bonn, Mathematical Institute, Bonn, Germany}
\email{lcote@math.uni-bonn.de}
\author{Yusuf Bar{\i}\c{s} Kartal}
\address {Department of Mathematics, National University of Singapore, Singapore}
\email {ybkartal@nus.edu.sg}

\date{}

\begin{document}
\begin{abstract}
We associate an invariant called the completed Tate cohomology to a filtered circle-equivariant spectrum and a complex oriented cohomology theory. We show that when the filtered spectrum is the spectral symplectic cohomology of a Liouville manifold, this invariant depends only on the stable homotopy type of the underlying manifold. We make explicit computations for several complex oriented cohomology theories, including Eilenberg--Maclane spectra, Morava K-theories, their integral counterparts, and complex K-theory. We show that the result for Eilenberg--Maclane spectra depends only on the rational homology, and we use the computations for Morava K-theory to recover the integral homology (as an ungraded group). In a different direction, we use the completed Tate cohomology computations for the complex K-theory to recover the complex K-theory of the underlying manifold from its equivariant filtered Floer homotopy type. A key Floer theoretic input is the computation of local equivariant Floer theory near the orbit of an autonomous Hamiltonian, which may be of independent interest.
\end{abstract}

\maketitle

\section{Introduction}
\subsection{Context} Given a Liouville manifold $M$, one can associate to it an invariant $SH(M)$ called \emph{the symplectic cohomology}. This is a version of Hamiltonian the Floer cohomology for open manifolds. It depends implicitly on a choice of coefficient ring as well as (purely topological) grading and orientation data.

In contrast to Hamiltonian Floer cohomology on closed manifolds, symplectic cohomology does not in general remember any information about the homotopy type of $M$. For example, any subcritical Weinstein manifold has vanishing symplectic cohomology. In particular, if $M= N \times \mathbb{C}$ for a Liouville manifold $N$, then it follows from \cite{cieliebak2002handle} that $SH(M)=0$. Therefore, the topology of $M$ can be arbitrarily complicated, while the symplectic cohomology vanishes.  

However, conjectures of Treumann \cite{treumann2019complex} inspired  by string theory and mirror symmetry suggest that the topological $K$-theory of the underlying manifold should be determined by Floer theoretic invariants such as the wrapped Fukaya category. This is impossible as we have seen above: vanishing symplectic cohomology implies vanishing wrapped Fukaya category. However, Treumann suggests that once the extra data of an \emph{action filtration} is remembered, one can reconstruct complex $K$-theory from these invariants.\footnote{Strictly speaking, Treumann's conjecture is concerned only with cotangent bundles \cite[p.\ 15]{treumann2019complex}.} 

Separately, the work of Albers--Cieliebak--Frauenfelder \cite{albers2016symplectic} and independently Zhao \cite{zhaoperiodic} shows that it is possible to recover rational information about the topology of $M$ by remembering extra structures on symplectic cochains. The relevant structures are (i) the $S^1$-action corresponding to ``loop rotation'' and (ii) the action filtration. \cite{albers2016symplectic} and \cite{zhaoperiodic} use these structures to build a version of ``symplectic Tate (co)homology'' with respect to the $S^1$-action. They prove that the resulting theory recovers the \emph{rational} cohomology of $M$; however, it is insensitive to torsion information. In particular, most of the information about the homology and the complex $K$-theory of $M$ is lost.

The goal of the present paper is to use methods of equivariant (Floer) homotopy theory (\cite{cotekartalequivariantfloerhomotopy}) and chromatic homotopy theory to recover much more information about the stable homotopy type of a Liouville manifold $M$. Surprisingly, one can recover from the filtered $S^1$-equivariant symplectic cohomology of $M$ over the sphere spectrum, not only the integral homology and the complex $K$-theory, but also the Morava $K$-theories of $M$, which can be seen as the building blocks of the entire stable homotopy type of $M$ (see \cite{barthelbeaudrychromatic} for a broad overview of the subject).  
%
%
%
%
%
\subsection{Extra structures on spectral symplectic cohomology}  Classically, symplectic cohomology is constructed as a colimit of Floer cohomology groups associated to a sequence of Hamiltonians with slope increasing to infinity. In \cite{largethesis}, Large lifted this construction to the sphere spectrum $\bS$ under appropriate topological assumptions. 

In \cite{cotekartalequivariantfloerhomotopy}, the authors further defined an $S^1$-equivariant model $SH_S(M,\bS)$ for symplectic cohomology, where the $S^1$-action corresponds to ``loop rotation''. Analogously to \cite{largethesis}, the authors first construct an $S^1$-equivariant spectrum  $HF_S(H,\bS)$ associated to a (family of) Hamiltonian(s) linear at infinity, and take a colimit as the slope goes to infinity. 

By choosing a cofinal sequence of Hamiltonians, one can naturally endow $SH_S(M,\bS)$ with the structure of a \emph{filtered equivariant spectrum}, a notion which will be made precise in \Cref{sec:filteredeqspec}. While the filtration depends on the choice of the sequence, two such sequences produce equivalent filtered spectra. We emphasize that $SH_S(M,\bS)$, viewed as a filtered equivariant spectrum, is a true invariant of $M$ as a Liouville manifold: it does not depend on any additional data such as  e.g.\ a Hamiltonian or a contact-type hypersurface. 

In summary, one can endow spectral symplectic cohomology with both an $S^1$-action corresponding to loop rotation, and with a (suitably interpreted) action filtration. 

\subsection{Summary of results}
To a complex oriented cohomology theory $R$ and a filtered $S^1$-equivariant spectrum $X$, we shall associate an invariant $\tate{R}(X)$ which we call the \emph{completed Tate cohomology}. The central result of this paper is the following computation:
\begingroup
\def\thethm{\ref*{thm:tatecohforsh}}
\begin{thm}
If $M$ is a stably framed Liouville manifold and $R$ is a complex oriented ring spectrum,
\begin{equation}
    \tate{R}(SH_S(M,\bS)) \simeq \tate{R}(\Sigma^\infty (M/\partial_\infty M)) 
\end{equation}
where $\Sigma^\infty (M/ \partial_\infty M)$ is endowed with the trivial filtration and trivial $S^1$ action.
\end{thm}
\addtocounter{thm}{-1}
\endgroup
In other words, $\tate{R}(SH_S(M,\bS))$ only depends on the homotopy type of $M$. Moreover, as we will see, $\tate{R}(\Sigma^\infty (M/\partial_\infty M))$ is a good approximation to $R^*(M,\partial M)$ for many $R$, including Morava $K$-theories. We refer to \Cref{subsection:proof-summary} for a summary of the proof of \Cref{thm:tatecohforsh}.

Let us briefly explain the definition of $\tate{R}(-)$, postponing the details to \Cref{subsection:completed-tate}. First of all, recall that the complex orientation on $R$ determines an isomorphism $R^*[[u]]=R^*(BS^1)$ (where $R^*=R^*(pt)$), and a formal group law $F_R\in R^*[[u_1,u_2]]$. The element $[n]_R(u)\in R^*[[u]]$ is, by definition, obtained by iterating $u$ under $F_R$, $n$ times. Given an $S^1$-equivariant filtered spectrum $X$ with filtration $X \simeq \colim(X_1\to X_2\to \dots)$, the completed Tate cohomology $\tate{R}(X)$ is defined as the limit of the $R^*((X_i)_{hS^1})$ localized at each $[n]_R(u)$ and $u$-adically completed. 

We compute $\tate{R}(SH_S(M,\bS))$ explicitly in a variety of examples. For the Eilenberg--Maclane spectra $R=H\bQ$, $H\bZ$, $H\bF_p$, we have the following analogues of \cite[Thm.\ 1.1(c) \& Sec.\ 5.1]{albers2016symplectic} and \cite[Thm.\ 1.1 \& Sec.\ 8]{zhaoperiodic}:
\begingroup
\def\thethm{\ref*{cor:eilenbergmaclanetatesh}}
\begin{cor}
We have $\tate{H\bQ}(SH_S(M,\bS))\simeq \tate{H\bZ}(SH_S(M,\bS))\simeq H_{2n-*}(M,\bQ((u)))$. On the other hand  $\tate{H\bF_p}(SH_S(M,\bS))\simeq 0$.
\end{cor}
\addtocounter{thm}{-1}
\endgroup

Note that for a commutative ring $A$, $\tate{HA}(SH_S(M,\bS))$ depends only on the $S^1$-equivariant (co)filtered homotopy type of $SH^*(M,A)$, which does not require Floer homotopy theory to define. For $A=\bQ$, this reaffirms the conclusion of \cite{zhaoperiodic, albers2016symplectic} that the symplectic cohomology with its $S^1$-action and filtration recovers the rational cohomology. For $A=\bZ$ this can be interpreted as a generalization of the computations for a disc and annulus in \cite[Sec.\ 8]{zhaoperiodic}. In this way, we lose the torsion information. See also \cite[Sec.\ 5.1]{albers2016symplectic} for related computations.

To access the torsion in the homology of $M$, we compute $\tate{R}(SH_S(M,\bS))$ for the (generalized) Morava $K$-theories $K_{p^k}(m)$:
\begingroup
\def\thethm{\ref*{cor:deformedmoravash}}
\begin{cor}
$\tate{K_{p^k}(m)}(SH_S(M,\bS))\simeq K_{p^k}(m)_{2n-*} (M)((u))$.
\end{cor}
\addtocounter{thm}{-1}
\endgroup
\Cref{cor:eilenbergmaclanetatesh} and \Cref{cor:deformedmoravash} also explain the term ``completed Tate cohomology'' (normally, to define Tate cohomology, one would localize $R^*((X_i)_{hS^1})$ at $[1]_R(u)=u$ only). For $R=H\bQ$, $K_{p^k}(m)$, this produces the same answer.

This can be used to show
\begingroup
\def\thethm{\ref*{thm:hlgrecovery}}
\begin{thm}
For $2(p^m-1)>dim_\bR(M)$ (or for $4(p^m-1)>dim_\bR(M)$ for Weinstein $M$), the groups $\tate{K_{p^k}(m)}(SH_S(M,\bS))$ and $H_{2n-*}(M,\bZ[v_m,v_m^{-1}]/p^k)((u))$ are isomorphic. 
\end{thm}
\addtocounter{thm}{-1}
\endgroup
Combining \Cref{thm:hlgrecovery} with the universal coefficients theorem, we show that the integral cohomology can be fully recovered:
\begingroup
\def\thethm{\ref*{cor:fullhlgrecovery}}
\begin{cor}
The filtered $S^1$-equivariant homotopy type of $SH_S(M,\bS)$ determines $H_*(M,\bZ)$.	
\end{cor}
\addtocounter{thm}{-1}
\endgroup
As noted, the invariant $\tate{R}(SH(M,\bS))$ depends on the filtration in a rather loose sense. Also, we choose to work with $SH(M,\bS)$ for simplicity, but the invariant $\tate{R}(SH(M,\bS))$ can be defined under milder assumptions on $M$, when $R$ is complex oriented. Our results can likely be extended to this generality but we will not pursue this. 

We also show we can recover complex $K$-theory:
\begingroup
\def\thethm{\ref*{thm:kutatesh}}
\begin{thm}
    After simultaneously completing at every prime $p$,  the groups $\tate{KU}(SH_S(M,\bS))$ and $KU_{2n-*} (M)((u))$ become isomorphic. In other words, $\tate{KU}(SH_S(M,\bS))^\wedge\simeq KU_{2n-*} (M)((u))^\wedge$. As a result, the filtered $S^1$-equivariant homotopy type of $SH_S(M,\bS)$ determines $KU_*(M)$ (and hence $KU^*(M)$ by the universal coefficients theorem). 
\end{thm}
\addtocounter{thm}{-1}
\endgroup

In other words \Cref{thm:kutatesh} says that the (co)filtered equivariant homotopy type of ``symplectic $K$-homology'' determines the $K$-theory of the underlying manifold. (Recall that for an abelian group $A$, $A^\wedge:=\lim_{n\in\bN}A/nA$.) This result is consistent with Treumann's proposal mentioned earlier; see \Cref{rk:treumann} for further discussions.

It is natural to ask whether it is possible to recover more information about $M$ from symplectic cohomology with coefficients in various ring spectra. 
\begin{question}\label{question:intro}
  Given a Liouville manifold $M$, does $SH_S(M; \mathbb{S})$ with its action filtration remember the stable homotopy type of $M$? 
\end{question}
As we mentioned in the beginning, the Morava $K$-theories are the building blocks of spectra. As a result, it may already be possible to obtain an affirmative answer to \Cref{question:intro} by combining the Floer theoretic input of this paper with other 
tools of chromatic homotopy theory, such as the chromatic tower and the chromatic fracture square.

\begin{rk}[Potential dynamical applications]
The constructions in this paper generalize to any autonomous Hamiltonian $H$ that is linear at infinity (with respect to any choice of Liouville form), and associate to it an $S^1$-equivariant spectrum. Moreover, the local Floer homology calculation \Cref{prop:localequiv} states that the $S^1$-equivariant homotopy type in a small action window near a non-constant orbit $X\cong S^1/C_k$ of $H$ is a Thom spectrum $X^\nu$, where $\nu$ is an $S^1$-equivariant virtual bundle over $X$. While the non-equivariant versions of this statement are straightforward (c.f. \cite{cieliebakhoferwysockiapplsymplchlg}), the $S^1$-equivariant version requires the use of new machinery.

It seems reasonable to expect that the local equivariant Floer cohomology calculation could be useful in dynamical applications. Namely, an easy calculation shows that the $S^1$-equivariant $R$-cohomology $R_{S^1}^*(X^\nu):=R_{S^1}^*((X^\nu)_{hS^1})$ is equal (up to a shift) to (i) $R^*[[u]]/([k]_R(u))$ if $\nu$ is trivial (i.e. $X$ is a \emph{good orbit}), (ii) $R^*[[u]]/([k]_R(u)/[k/2]_R(u))$ is $\nu$ is non-trivial (i.e. $X$ is a \emph{bad orbit}, note that $k$ is even in this case). Therefore, the local equivariant Floer cohomology retains information about the multiplicity $k$ of the orbits, as well as the bad orbits. For instance, if one ignores the bad orbits, and considers the kernel of the localization at $u$, one discards all multiple orbits, and is left with the simple periodic orbits. One expects this to be helpful in answering questions related to simple periodic orbits (a major example of such a question is Hofer-Wysocki-Zehnder's two or infinity conjecture, which is recently proven in \cite{dancgtwoinfty}). In a different direction, the same calculation shows that bad orbits do not contribute to the equivariant Floer cohomology when $2\in R^*$ is invertible, and they behave differently than the good orbits when $R$ is $2$-local (e.g. $R=H\bF_2$ or $K_2(n)$). Therefore, it is reasonable to anticipate that this calculation will be useful in addressing questions related to the existence of bad orbits.
\end{rk}

\subsection*{Acknowledgements} 
The authors would like to thank Irakli Patchkoria, Tobias Barthel, Andrew Blumberg, Sanath Devalapurkar, Yash Deshmukh, Kaif Hilman, Jonas McCandless, Ivan Smith, Marco Volpe and Hiro Tanaka for useful discussions and comments. The authors would also like to thank Tim Large for a particularly inspiring seminar talk, from which they learned that equivariant localization works for Morava K-theory (see \Cref{exmp:ordinarymorava}). Finally, the authors thank the anonymous referee for many useful comments.

This work was completed while the first author was visiting the Max Planck Institute for Mathematics, which he thanks for its hospitality. The first author was also partially supported by the National Science Foundation under Grant No. DMS-2305257 and by the Simons Collaboration on Homological Mirror symmetry; and the second author was supported by ERC Starting Grant 850713 (Homological mirror symmetry, Hodge theory, and symplectic topology), and the Simons Collaboration on Homological Mirror Symmetry.

\section{Filtered equivariant spectra and completed Tate cohomology}\label{sec:filteredeqspec}
\subsection{Reminders in homotopy theory}

This paper considers Floer theoretic invariants defined over ring spectra. Since the modern foundations for ring spectra are written using infinity categorical language, some amount of this language is unavoidable in our paper. However, we have included some explanatory remarks (e.g. \Cref{remark:symplectic-context}) for geometrically minded readers. 

\subsubsection{Conventions} 
Following the standard convention, we will often write category when we mean infinity category, colimit when we mean homotopy colimit, etc. 

Let $\Sp$ be the infinity category of spectra. This is a symmetric monoidal category; we typically denote the monoidal product (i.e.\ the smash product) by $-\otimes -$ and denote the internal hom (i.e.\ the function spectrum) by $F(-, -)$.  A \emph{ring spectrum} (or \emph{multiplicative cohomology theory}) is a spectrum $R \in \Sp$ equipped with a multiplication $R \otimes R \to R$ which is homotopically associative and unital. Following the usual convention, we will not distinguish between generalized cohomology theories and their representing spectra. 

Let $\Top$ be the (infinity) category of spaces and let $\Top_*$ be the (infinity) category of pointed spaces. There are natural functors $(-)_+: \Top \to \Top_*$ and $\Sigma^\infty(-): \Top_* \to \Sp$ (left adjoint to the forgetful functor, resp.\ to $\Omega^\infty(-)$). We typically denote their composition by $\Sigma^\infty (-)_+$.

Given a spectrum $X$, its $R$-cohomology is the graded abelian group $R^*(X):= \pi_{-*} F(X, R)$. If $Y \in \Top_*$, we let $R^*(Y):= R^*(\Sigma^\infty Y)$. If $Z \in \Top$, then $R^*(Z):= R^*(\Sigma^\infty Z_+)$. We define the reduced cohomology $\tilde{R}^*(Z):= \ker(R^*(Z_+) \to R^*(*_+))$, where $*_+ \to Z_+$ collapses $*_+$ to the basepoint $Z_+$.  We let $R^*:= R^*(\mathbb{S})= \pi_{-*}R$ be the \emph{coefficients} of the spectrum $R$.  We adopt analogous conventions for $R$-homology.




\subsubsection{Borel $G$-spectra} Let $G$ be a compact Lie group. The functor category $\Fun(BG, \Sp):= \Sp^{BG}$ shall be called the category of \emph{Borel $G$-spectra}. It inherits a symmetric monoidal structure from $\Sp$, where the monoidal unit is the sphere with the trivial $G$-action. We will often refer to the objects of $\Sp^{BG}$ as \emph{$G$-equivariant spectra}.\footnote{There are other non-equivalent objects in homotopy theory which also deserve this name, but they will never be considered in this paper.} 

Given a Borel $G$-spectrum $X: BG \to \Sp$, we let 
\begin{equation}\label{equation:fixed-orbits}
    X_{hG}:= \colim_{BG} X \in \Sp \hspace{1cm} X^{hG}:= \lim_{BG} X \in \Sp
\end{equation} 
be respectively the \emph{homotopy quotient} (or \emph{homotopy orbits}) and the \emph{homotopy fixed points}. More concretely, $X_{hG}$ can be defined as $X\otimes_G EG_+$, and $X^{hG}$ can be defined as $Map^G(EG_+,X)$. It is often also useful to note that $(-)_{hG}$ and $(-)^{hG}$ are the left (resp.\ right) adjoints of the pullback $\pi^*: \Sp \to \Sp^{BG}$ of the map $\pi: BG \to *$ giving a spectrum the trivial $G$ action. Note finally the identity $F(X_{hG}, R)\simeq F(X, R)^{hG}$, which follows from \eqref{equation:fixed-orbits} by abstract nonsense. 

\begin{lem}\label{lem:module1}
If $R$ is an $E_k$ ring spectrum, $k \geq 1$, then $F(\bS_{hG}, R)\simeq F(\Sigma^\infty  BG_+, R)$ is also an $E_k$ ring spectrum. For any $X \in \Sp^{BG}$, $F(X_{hG}, R) $ is a module over $F(\bS_{hG}, R)$.
\end{lem}
\begin{proof}
By virtue of being the monoidal unit, the sphere $\bS$ with trivial $G$-action forms a co-commutative co-algebra. By virtue of being a left-adjoint of a lax monoidal functor, $(-)_{hG}$ is oplax monoidal, i.e. there are natural maps $(X\otimes Y)_{hG}\to X_{hG}\otimes Y_{hG}$ (satisfying some higher coherence conditions \cite{haugseng2023lax, torii2023perfect}). Hence $\bS_{hG}$ is a co-commutative co-algebra, and we have canonical maps $F(\bS_{hG}, R) \otimes F(\bS_{hG}, R) \to F(\bS_{hG} \otimes \bS_{hG}, R \otimes R) \to F(\bS_{hG},R \otimes R) \to F(\bS_{hG}, R)$. Similarly, $X_{hG}$ is a co-module over $\bS_{hG}$, leading to a $F(\bS_{hG}, R)$-module structure on $F(X_{hG}, R)$.
\end{proof}

By similar considerations, one can argue that if $Y$ is a $G$-space, then the diagonal makes $\Sigma^\infty Y_+$ into a co-algebra, and hence $F(Y_{hG}, R)$ is an algebra.

The cohomology $R^*(X_{hG})= \pi_*F(X_{hG}, R)$ is the \emph{$R$-valued equivariant cohomology}. Note that if $X$ is a \emph{space}, then $\Sigma^\infty (X_{hG})_+= (\Sigma^\infty X_+)_{hG}$ and we have $R^*((\Sigma^\infty X_+)_{hS^1})= R^*(X_{hG})= R^*(EG \times_G X)$.  

\subsubsection{Complex orientations}

\begin{defn}
    A ring spectrum is said to be \emph{complex orientable} if the pull-back $R^*(BS^1)=R^*(\bC\bP^\infty)\to R^*(\bC\bP^1)$ is surjective. A complex orientation is a lift of $1\in \pi_0(\bC\bP^1)\cong \tilde R^2(\bC\bP^1)$ to $R^2(BS^1)$. A choice of such a lift specifies an isomorphism $R^*(BS^1)\cong R^*[[u]]$, where $u$ is a variable of degree $-2$.
\end{defn}

A complex orientation defines a formal group law $F_R\in R^*[[u_1,u_2]]$. Concretely, $F_R$ is the image of $u$ under the pull-back $R^*[[u]]=R^*(BS^1)\to R^*[[u_1,u_2]]=R^*(BS^1\times BS^1)$ of the classifying map $BS^1\times BS^1\to BS^1$ corresponding to the tensor product of universal line bundles. One can think of $u\in R^*[[u]] \cong R^*(BS^1)$ as the analogue of the first Chern class in the generalized cohomology theory $R$, and $F_R$ is the formula for the Chern class of tensor products of line bundles. 

Let $[n]_R(u)\in R^*[[u]]$ denote the $n^{th}$ iterate of $u$ under the formal group law $F_R$ (i.e. $[2]_R(u)=F_R(u,u)$, $[3]_R(u)=F_R([2]_R(u),u)$, $\dots$). 

\begin{lem}\label{lemma:0-action}
Let $Y:=S^1/C_k$ denote the homogeneous $S^1$-space with stabilizer given by the cyclic group $C_k, k> 0$.  Then $[k]_R(u) \in R^*(\mathbb{S}_{hS^1})= R^*[[u]]$ acts by zero on $R^*(Y_{hS^1})$, under the $R^*(\mathbb{S}_{hS^1})$-algebra structure induced by the map of $G$-spaces $Y \to *$.

\end{lem} 
\begin{proof}
    Note that $Y_{hS^1}\simeq BC_k$. It follows from the Gysin sequence 
    \begin{equation}
        R^*[[u]] \xrightarrow{ \cdot [k]_R(u)} R^*[[u]] \to R^*(BC_k) \to
    \end{equation} associated to the fibration $S^1 \hookrightarrow BC_k \to BS^1$ that $[k]_R(u)$ pulls back to zero in $R^*(Y_{hS^1})$. 
\end{proof}

\subsection{Filtered objects of stable $\infty$-categories}

In this section, we discuss filtrations on objects of stable $\infty$-categories. For motivation, geometrically minded readers may want to start with \Cref{remark:symplectic-context}.

Let $\mathcal{C}$ be a stable $\infty$ category. Then one can form the (infinity) category $\operatorname{Fun}(\mathbb{N}, \mathcal{C})$ of functors from (the nerve of) $\mathbb{N}$ to $\mathcal{C}$. An object is just a sequence $X_1 \to X_2 \to \dots$, and a $(1-)$morphism is a diagram of such sequences, etc. We write $\widetilde{\mathcal{C}}_{fil}:= \operatorname{Fun}(\mathbb{N}, \mathcal{C})$.

Given a filtration $X_1 \to X_2 \to \dots$ and a sequence $1\leq i_1< i_2< \dots $, there is a corresponding  \emph{subfiltration} $X_{i_1}\to X_{i_2}\to \dots$. There is a ``tautological'' morphism $(X_1 \to X_2 \to \dots ) \to (X_{i_1}\to X_{i_2}\to \dots)$ in $\widetilde{\mathcal{C}}_{fil}$. 

\begin{defn}\label{definition:equivalence}
Let $\mathcal{C}_{fil}$ be obtained from $\widetilde{\mathcal{C}}_{fil}$ by localizing at the tautological morphisms corresponding to every subfiltration. 
\end{defn}

The (homotopy) colimit defines a functor $\widetilde{\mathcal{C}}_{fil} \to \mathcal{C}$. By elementary properties of the colimit, it passes to a functor $\mathcal{C}_{fil} \to \mathcal{C}$.

A \emph{filtration} on an object of $X \in \mathcal{C}$ is simply a lift to $\mathcal{C}_{fil}$. We also call such a lift a \emph{filtered object}.
Two filtered objects are said to be \emph{equivalent} if they are equivalent in $\mathcal{C}_{fil}$. We will occasionally denote a filtered object by $(X_1 \to X_2 \to \dots \simeq X)$. 

Objects in $\mathcal{C}$ may have multiple non-equivalent filtrations: 
\begin{exmp}
Consider the sequence of $\bQ[u]$-modules given by $\dots\to \bQ[u]/(u^3)\to\bQ[u]/(u^2)\to \bQ[u]/(u)$. This is a filtration on $\bQ[[u]]$ in the category $\cC=Mod(\bQ[u])^{op}$. If this was equivalent to the trivial filtration, we would have
\begin{equation}
0=	\lim\limits_{k} \big(\bQ[u]/(u^k)\otimes_{\bQ[u]}^L\bQ[u,u^{-1}]\big)=\lim\limits_{k} \big(\bQ[[u]]\otimes_{\bQ[u]}^L\bQ[u,u^{-1}]\big)\neq 0 
\end{equation}
\end{exmp}
A concrete way to show that two filtrations $X_1\to X_2\to\dots $ and $X_1'\to X_2'\to\dots $ are equivalent is to find sequences $1\leq i_1\leq i_2 \dots $, $1\leq j_1\leq j_2 \dots $, and maps $X_k'\to X_{i_k}$, $X_k\to X_{j_k}'$ of filtrations such that the two sided compositions are equivalent to the natural subfiltration maps $X_k\to X_{i_{j_k}}$ and $X'_k\to X'_{j_{i_k}}$. One can also show that, for any sequence $1\leq i_1\leq i_2\leq \dots $ that is not strictly monotone, but $i_k\to\infty$, $X_{i_1}\to X_{i_2}\to \dots$ is still equivalent to $X_1\to X_2\to\dots $ (by a zigzag through a common subfiltration). Throughout the paper, unless stated otherwise, we only consider the filtrations $X_1\to X_2\to \dots$ such that all of $X_i$ are of finite type. 

\begin{rk}[Symplectic context]\label{remark:symplectic-context}
We briefly sketch how this abstract formalism is related to a very classical construction in symplectic geometry. 

Let $\mathcal{C}$ (temporarily) be the category of chain complexes. To define symplectic cohomology, one typically proceeds in two steps. First: consider a cofinal sequence of Hamiltonians $H_1, H_2, \dots$ giving rise to a sequence $CF(M; H_1) \to CF(M; H_2) \to \dots$, where $CF(M; H_i) \in \mathcal{C}$.\footnote{ Here we omit the choice of $J$ and ignore questions of genericity/transversality/etc} Second: take the colimit.  The \emph{proof} that the colimit does not depend on the choice of cofinal sequence of Hamiltonians (see e.g.\ \cite[(4a)]{seidel2008biased}) actually shows more, namely that $SH(M)$ is well-defined in $\mathcal{C}_{fil}$. 

Using the Borel construction and the foundations from \cite{largethesis, cotekartalequivariantfloerhomotopy}, one can define (under appropriate assumptions; see \Cref{subsubsection:data-borel}) $CF(M; H_i)$ as a filtered spectrum with an $S^1$-action; that is, one can take $\mathcal{C}= \operatorname{Sp}^{BS^1}$. Hence $SH(M)$ can be defined as an object in $\Sp^{BS^1}_{fil}$. 
\end{rk}
\begin{rk}
Commutative diagrams in any infinity categorical context are secretly understood to be homotopy coherent, meaning that arrows commute up to a coherent system of homotopies. Such subtleties are irrelevant for the purpose of defining symplectic cohomology as a filtered spectrum, because any functor $\mathbb{N} \to \operatorname{Ho}(\mathcal{C})$, for $\mathcal{C}$ an infinity category, lifts to a homotopy coherent diagram in an essentially unique way. 
\end{rk}

\subsection{Completed Tate cohomology}\label{subsection:completed-tate}
\subsubsection{The definition}

Given a complex-oriented cohomology theory $R$, we define
\begin{equation}
    \mathfrak{F}_R:= \{ [1]_R(u),[2]_R(u),[3]_R(u),\dots \} \subset R^*[[u]].
\end{equation}
By definition, $[k]_R(u)$ is homogeneous of degree $-2$. 

\begin{defn}\label{defn:completedtate}
Let $X \in \Sp^{BS^1}_{fil}$ be a filtered $S^1$-equivariant spectrum. Given a complex-oriented cohomology theory $R$, define the \emph{completed Tate cohomology} of $X$ to be the limit
\begin{equation}\label{eq:completedtate}
\tate{R}(X):=\lim_k	\mathfrak{F}_R^{-1}R^*((X_k)_{hS^1})=\lim_k	R^*((X_k)_{hS^1})[[1]_R(u)^{-1},[2]_R(u)^{-1},[3]_R(u)^{-1}\dots ]
\end{equation}
where the localization is completed with respect to the topology induced by the $u$-adic topology on $R^*((X_k)_{hS^1})$. This is different from the $u$-adic topology on the localization, which is coarse. See \Cref{note:topomod}.
\end{defn}
\begin{note}\label{note:topomod}
Recall that the completion of a filtered commutative group $M$ is defined to be $\lim_p M/F^pM$. This is equivalent to the topological completion for the topology defined by the neighborhood basis $\{F^pM\}$. If $M$ is graded, one takes the limit in the graded category too; thus, the completion is still graded. If $A$ is a commutative topological algebra, $S\subset A$ and $M$ is a topological module over $A$, then the localization $S^{-1}M$ inherits a topology. More precisely, $S^{-1}M$ is a colimit over a diagram with terms $M$, and the arrows given by multiplication with the elements of $S$. This induces a topology on $S^{-1}M$: explicitly a subset $U\subset S^{-1}M$ is open if and only if the set $\{m\in M:s^{-1}m\in  U\}$ is open for every $s\in S$. If $u\in A$, and $A$ is endowed with the topology generated by $u^kA$ (i.e. the $u$-adic topology), and $u\in S$, then the topology induced on $S^{-1}M$ is not the $u$-adic topology of the quotient ($u^kS^{-1}M=S^{-1}M$, so the $u$-adic topology is coarse).

For example, $\bQ[q]$ is filtered by $q^p\bQ[q], p\geq 0$, and this topology descends to $\Q[q,q^{-1}]$. Notice however, it is not generated by $q^p\bQ[q,q^{-1}], p\geq 0$, but rather by $q^p\bQ[q]\subset \bQ[q,q^{-1}]$, which also is the topology corresponding to the $u$-adic norm. The completion with respect to this topology is given by $\bQ((u))$. 
\end{note}

Observe
\begin{itemize}
	\item $\tate{R}(X)$ carries the structure of a complete $\tate{R}(\mathbb{S})$-module.
	\item $\tate{R}(X)$ is functorial in $X$ and sends finite colimits to limits. 
	\item $\tate{R}(X)$ is functorial in $R$. In other words, if $R\to R'$ is an homomorphism of complex oriented ring spectra (respecting complex orientations), then there is a natural map $\tate{R}(X)\to \tate{R'}(X)$. 
\end{itemize}
Be warned that localization does not commute with infinite limits, so one cannot in general pass the limit inside the localization in \eqref{eq:completedtate} (this is ultimately what makes the invariant useful). On the other hand, if the filtration is trivial/finite, then there is not need to pass to the limit. Hence we will also use the notation $\tate{R}(X_k)$ to refer to the terms in the limit in \eqref{eq:completedtate}.

\subsubsection{Completed Tate cohomology of the building blocks}

In this section, we compute the completed Tate cohomology of Thom spectra of bundles over homogeneous spaces over $S^1$, as well as spaces with trivial $S^1$-action. These are the basic building blocks of the spectra we are interested in. First
\begin{lem}\label{lem:nontrivaction}
Consider $S^1/C_k$, the homogeneous $S^1$-space with stabilizer given by the cyclic group $C_k, k> 0$. Consider a virtual equivariant vector bundle over $S^1/C_k$ and let $X$ denote the corresponding Thom spectrum. Then $\tate{R}(X)=0$. 
\end{lem}
\begin{proof}
Write $Y$ for the space $S^1/C_k$. A virtual equivariant vector bundle can be written as $E-\underline{V}$, where $E$ is an equivariant vector bundle and $\underline{V}$ is a topologically trivial equivariant bundle corresponding to a representation $V$ of $S^1$, see \cite{segalequivariantktheory}. The Thom spectrum is given by 
\begin{equation}
X=\Sigma_V^{-1}	Y^E=\bS^{-V}\otimes Y^E
\end{equation}
Similarly to \Cref{lem:module1}, $X$ is a comodule over $\Sigma^\infty Y_+$. More precisely, the diagonal $Y\to Y\times Y$ induces a map $Y^E\to Y^E\otimes \Sigma^\infty Y_+$ (to see this note that $E \boxplus 0$ pulls back to $E$ under the diagonal and use the functoriality of Thom spectra). This immediately extends to a coassociative coproduct 
\begin{equation}
X=	\bS^{-V}\otimes Y^E\to \bS^{-V}\otimes Y^E\otimes \Sigma^\infty Y_+=X\otimes \Sigma^\infty Y_+
\end{equation}
and as in the proof of \Cref{lem:module1}, this induces an $R^*(Y_{hS^1})$-module structure on $R^*(X_{hS^1})$. On the other hand, $Y_{hS^1}=BC_k$ and \Cref{lemma:0-action} ensures that $[k]_R(u)\in R^*[[u]]$ acts on $R^*(Y_{hS^1})$ by $0$. 
%
As a result, the $R^*[[u]]$-module structure on $R^*(X_{hS^1})$ refines to a $R^*[[u]]/([k]_R(u))$-module structure, and inverting $[k]_R(u)$ kills the group, and $\mathfrak{F}_R^{-1}R^*(X_{hS^1})=0$. 
%
\end{proof}

We also have
\begin{lem}\label{lem:trivaction}
If $X$ is a finite spectrum with trivial $S^1$-action, then 
\begin{equation}
\tate{R}(X)\simeq \mathfrak{F}_R^{-1}R^*(X)[[u]]=R^*(X)[[u]][[1]_R(u)^{-1},[2]_R(u)^{-1},[3]_R(u)^{-1}\dots ]
\end{equation}
where the localization is in the $u$-completed as before. 
\end{lem}
More concisely, we can write this as $\tate{R}(X)\simeq \tate{R}(\bS)\otimes_{R^*}R^*(X)$. Since $X$ is assumed to be finite, tensor product is automatically complete. Otherwise, one would need to use the completed tensor product. 
\begin{proof}
It is easy to check that $R^*(X_{hS^1})=R^*(X\wedge \Sigma^\infty BS^1_+)\simeq R^*(X)[[u]]$. The result follows from this.
\end{proof}

\subsubsection{Computations for the trivial action}
In this section, we compute the completed Tate cohomology $\tate{R}(X)$ of a finite spectrum $X$ with trivial action and filtration for various $R$.

\begin{exmp}\label{exmp:eilenmac}
Assume $R=HA$, the Eilenberg--Maclane spectrum of a commutative ring $A$. Then, $R^*=A$ and $F_R(u_1,u_2)=u_1+u_2$, the \emph{additive formal group law}. Therefore, $[n]_R(u)=nu$, and inverting it is the same as inverting $n$ and $u$. As a result, if $X$ is endowed with a trivial action and filtration, then $\tate{R}(X)=H^*(X,A\otimes_\bZ \bQ)((u))$. In particular, if $n=0$ in $A$, for some $n\geq 2$, then $\tate{R}(X)=0$. On the other hand, if $R=H\bQ$ or $H\bZ$, then $\tate{R}(X)=H^*(X,\bQ)((u))$. 
\end{exmp}
If in \eqref{eq:completedtate} we were not taking the $u$-adically completed localization, $\tate{H\bZ}(X)$ and $H^*(X,\bQ)((u))$ would match only after $u$-adically completing the former.

An important example is the following:
\begin{exmp}\label{exmp:ordinarymorava}
Let $R=K_p(n)$ be a Morava $K$-theory. Then $R^*=\bF_p[v_n,v_n^{-1}]$. As $[k]_R(u)$ is always of the form $ku+h.o.t.$, when $k$ is coprime with $p$, $[k]_R(u)$ is $u$ multiplied with a unit. Moreover, $[p]_R(u)=v_nu^{p^n}$ (see \cite[Theorem 1.3]{wurglermorava} for instance). Hence, by iteration, we see that $[k]_R(u)$ is the same as a power of $u$ multiplied by a unit (if $p^l||k$, $[k]_R(u)=\frac{k}{p^l}v_n^{\frac{p^{nl}-1}{p^n-1}}u^{p^{nl}}+h.o.t.$).  Therefore inverting all $[k]_R(u)$ is equivalent to inverting $u$. In other words, for $X$ with a trivial action and trivial filtration $\tate{K_p(n)}(X)=K_p(n)^*(X)((u))$.  
\end{exmp}
\begin{exmp}
More generally, assume $F_R$ has height exactly $n$, i.e. $[p]_R(u)$ is $u^{p^n}$ multiplied by a unit. Also assume $p$ is nilpotent in $R^*$. Therefore, if $gcd(k,p)=1$, $k$ is invertible in $R^*$. Also, as we argued in \Cref{exmp:ordinarymorava}, $[k]_R(u)=ku+h.o.t.$, and invertibility of $k$ implies that $[k]_R(u)$ is equal to $u$ times a unit. Moreover, for a general $k$ such that $p^l||k$ (i.e. $l$ is the highest dividing power of $p$ in $k$), we have
\begin{equation}
    [k]_R(u)=[\frac{k}{p^l}]_R([p^l]_R(u))= [\frac{k}{p^l}]_R([p]_R([p]_R([p]_R(\dots[p]_R(u)\dots)))),
\end{equation}
i.e. $[k]_R(u)$ is obtained by an iterative application of $[\frac{k}{p^l}]_R(u)$ and of $[p]_R(u)$ $l$ times. As a result, $[k]_R(u)$ is equal to a power of $u$ multiplied by a unit. 
Therefore, inverting $u$ suffices to invert all $[k]_R(u)$. This implies $\tate{R}(X)=R^*(X)((u))$.
\end{exmp}
\begin{exmp}
Let $R=k_p(n)$ denote the connective Morava $K$-theory, which satisfies $R^*=\bF_p[v_n]$ and $[p]_R(u)=v_nu^{p^n}$. The calculation of $[k]_R(u)$ is the same as in \Cref{exmp:ordinarymorava}, but this time inverting these elements require inverting $u$ and $v_n$. Therefore, $\tate{k_p(n)}(X)=K_p(n)^*(X)((u))$, which is analogous to $H\bZ$ leading to rational cohomology in \Cref{exmp:eilenmac}. 
\end{exmp}
The following example will also be important in the subsequent sections:
\begin{exmp}\label{exmp:deformedmorava}
Let $\widetilde K_p(n)$ be the integral Morava $K$-theory, i.e. the complex oriented spectrum satisfying $\widetilde K_p(n)^*=\bZ_p[v_n,v_n^{-1}]$ and $[-p]_{\widetilde K_p(n)}(u)=-pu+v_nu^{p^n}$. 
Let $K_{p^k}(n)$ denote its quotient by $p^k$; see \cite[Proposition 5.14]{abouzaidmcleansmithglobal1}. For $R=K_{p^k}(n)$, the same formula $[-p]_R(u)=-pu+v_nu^{p^n}$ holds. For $p\not|m$, $[mp^l]_R(u)$ is of the form $u^{p^{nl}}+ O(p)+o(u^{p^{nl}})$ times a unit, where $O(p)$ denotes a multiple of $p$ and $o(u^{p^{nl}})$ denotes the sum of the terms such that the exponent of $u$ is larger than $p^{nl}$. One can also write this as $u^{p^{nl}}+ O(p)$ times a unit (as $u^{p^{nl}}+o(u^{p^{nl}})$ is $u^{p^{nl}}$ times a unit). As $p$ is nilpotent, after inverting $u=[1]_R(u)$, elements of the form $u^{p^{nl}}+ O(p)$ automatically become units. As a result, for $X$ with a trivial action and filtration, $\tate{R}(X)=R^*(X)((u))$. The same holds for $R=\widetilde K_p(n)$ after $p$-completing both sides, i.e. $\tate{R}(X)_p^\wedge=R^*(X)((u))_p^\wedge$. The reason is that in this case, where $m=p^lm_0$, where $p$ does not divide $m_0$, $[m]_R(u)$ is given by a power of $u$ times a unit times $m_0$ plus $O(p)$. After $p$-completion, this becomes the same as a power of $u$ times a unit. Hence, since $u$ is invertible, all $[m]_R(u)$ are.
\end{exmp}
The following example will be important in recovering information about complex $K$-theory from Floer theory: 
\begin{exmp}\label{exmp:ku}
Recall that $KU$, the complex $K$-theory spectrum, has $KU^*=\bZ[\beta,\beta^{-1}]$, where $|\beta|=2$ and $F_{KU}(u_1,u_2)=u_1+u_2+u_1u_2=(1+u_1)(1+u_2)-1$, the \emph{multiplicative formal group law}.
Let $p$ be a prime and $R=KU/p^k$ so that $R^*=\bZ/p^k[\beta,\beta^{-1}]$ and $[p^l]_R(u)={p^l\choose 1}u+{p^l\choose 2}u^2+\dots+u^{p^l}$, which implies $[mp^l]_R(u)=u^{p^l}+O(p)$ times a unit as in \Cref{exmp:deformedmorava}, for $p\not|m$. As a result, inverting $u$ suffices as before and for $X$ with a trivial action and filtration, $\tate{R}(X)=R^*(X)((u))$. The same holds for $KU$ after simultaneous completion of both sides at each $p$, i.e. $\tate{KU}(X)^\wedge=KU^*(X)((u))^\wedge$, where for an abelian group $A$, $A^\wedge$ denotes the inverse limit of $A/nA$, $n\in\bN$. 
\end{exmp}
One can make such concrete calculations for other complex oriented spectra, such as Brown--Peterson spectrum (or its variants such as its localization at $v_1\in BP^*$). This could possibly let one to recover information about $MU$ from Floer theory.


\section{$S^1$-equivariant spectral symplectic cohomology}
\subsection{Conventions}
Let $(M, \omega= d\lambda)$ be an exact symplectic manifold. Set $S^1= \bR/\bZ$. Given a Hamiltonian $H: S^1 \times M \to \bR$, we write $H_t(-)= H(t,-)$. The Hamiltonian vector field is defined by $\omega(X_t, -)= -dH_t$.  The $H$-perturbed action functional is the map
\begin{align}\label{equation:action-functional}
    \cA_H: \mathcal{L}M &\to \bR \\
    x &\mapsto \int_{S^1} (x^*\lambda - H_t(x(t)) )dt 
\end{align}
This is the opposite of conventions of \cite{abouzaid2013symplectic,seidel2008biased}. 
Given an auxiliary (possibly time-dependent) almost-complex structure $J$, we shall consider solutions $u: \bR \times S^1 \to M$ to Floer's equation
\begin{align}\label{equation:floer}
    \partial_s u + J(\partial_t u - X_{H_t})=0 
\end{align}
which are asymptotic to closed trajectories $x^\pm= \lim_{s \to \pm \infty}u(s,-)$ of the Hamiltonian $H_t$. We call $x^+$ the input and $x^-$ the output. Solutions to Floer's equation \eqref{equation:floer} are formally negative gradient trajectories of \eqref{equation:action-functional}, and we have the identity $\cA_H(x_-)  \leq \cA_H(x_+)$. In other words, the Floer trajectories decrease the action. 
More generally, if $(H^s)_{s \in \bR}$ is a family of Hamiltonians which is independent of $s$ for $s \in \bR-(0,1)$ and satisfies $\partial_s H^s \leq 0$, we can generalize \eqref{equation:floer} to the \emph{continuation equation} $\partial_s u+ J(\partial_t u - X^{H^s_t})=0$ and we have $\cA_{H_0}(x_-) \leq \cA_{H_1}(x_+)$, i.e. the continuation trajectories also decrease the action. 

\subsection{Liouville manifolds}
Recall that a \emph{Liouville manifold} $(M, \lambda)$ is an exact symplectic manifold with the property that there exists an embedding 
\begin{equation}
    \iota:(S^+(\partial_\infty M), \lambda_{can}) \hookrightarrow (M, \lambda), \iota^*\lambda= \lambda_{can}
\end{equation} 
of the positive symplectization of a co-oriented contact manifold $(\partial_\infty M, \xi_\infty)$ which covers a neighborhood of infinity.\footnote{Given a contact manifold $(Y, \xi)$, its symplectization $SY:= \{ \alpha \in T^*Y \mid \alpha(\xi) \neq 0\}$ is an exact symplectic manifold with respect to the canonical $1$-form $\lambda_{can}$. We say that $(Y, \xi)$ is co-orientable if $SY$ is disconnected, and a co-orientation amounts to a choice of connected component $S^+Y$ which we call the positive symplectization. } A choice of contact form $\ker \lambda_\infty= \xi_\infty$ on $\partial_\infty M$ induces a decomposition
\begin{equation}\label{equation:liouville-decomposition}
	M= \overline{M} \cup_{\partial \overline{M}} [1, \infty) \times \partial_\infty M,
\end{equation} 
where $\lambda$ restricts on $[1, \infty) \times \partial_\infty M$ to $r\lambda_\infty$. The coordinate $r \in [1, \infty)$ shall be called the \emph{Liouville parameter} and the vector field $Z$ on $M$ defined by the property $i_Z d\lambda = \lambda$ shall be called the \emph{Liouville vector field}. Note that for a given Liouville manifold $(M, \lambda)$, the data of contact form $\ker \lambda_\infty= \xi_\infty$ is equivalent to the data of a contact hypersurface transverse to $Z$. 

A Hamiltonian $H: M \to \mathbb{R}$ is said to be \emph{linear at infinity} if $Z H = H$ outside a compact set. Any Hamiltonian which satisfies $H(r, m)= C r$ outside a compact set, with respect to the (non-canonical) decomposition \eqref{equation:liouville-decomposition}, is manifestly linear at infinity. The constant $C$ is called the \emph{slope} of $H$. We emphasize that the notion of slope only makes sense once we have fixed a decomposition \eqref{equation:liouville-decomposition}, i.e.\ a contact form $\ker \lambda_\infty= \xi_\infty$. For instance, scaling $\lambda_\infty$ scales the slope by a corresponding amount. 

If $R$ is the Reeb vector field on $(\partial_\infty M, \lambda_\infty)$ and $H(r, m)= h(r)$, then $X_H= h'(r) R$ on the collar $[1, \infty) \times \partial_\infty M$. So closed time-$1$ orbits of $H$ on the level $\{r=r_0\} \subset [1, \infty) \times \partial_\infty M$ correspond to closed Reeb orbits of length $h'(r_0)$. The action of such an orbit is $r_0h'(r_0) -h(r_0)$.

A Liouville manifolds is said to be \emph{stably framed} if it is equipped with a fixed isomorphism $T_M\oplus\bC^k\cong\bC^{N+k}$ of symplectic vector bundles, for some $k\geq 0$. \emph{For the remainder of this paper, we restrict our attention to stably framed Liouville manifolds.} For this reason, we will typically drop the adjective ``stably framed''. The stable framing will be used to define $SH(M,\bS)$ and $SH_S(M,\bS)$. In fact, we will mostly need symplectic cohomology with coefficients in complex oriented spectra, an invariant which can be defined under weaker assumptions on $M$. 

\subsection{Review of (spectral) symplectic cohomology}\label{subsection:spectral-sh}
The construction of symplectic cohomology over the sphere spectrum was first carried out by Large \cite{largethesis}. However, our presentation will rather follow \cite{cotekartalequivariantfloerhomotopy} which follows different conventions.

Let $M$ be a Liouville manifold. Let $H: S^1 \times M \to \bR$ be a Hamiltonian which is non-degenerate and linear at infinity. Let $J: S^1 \to \mathcal{J}(M)$ be a generic family of cylindrical almost-complex structures. To this data, one can associate a \emph{flow category} $\cM_{H, J}$ such that
\begin{itemize}
	\item $ob(\cM_{H,J})$ is the set of $1$-periodic orbits of $H$
	\item the morphisms $\cM_{H,J}(x,y)$ is given by the moduli of Floer trajectories from $x$ to $y$ and it forms a manifold with corners with a stable trivialization of its tangent bundle
	\item the composition maps are smooth embeddings into boundary strata, which is covered by the images of the composition maps
\end{itemize}

By out standing assumption that $M$ is stably framed, $\cM_{H, J}$ can be naturally enhanced to a framed flow category.

To the frame flow category $\cM_{H,J}$, one associates a spectrum $HF(H; \bS):= |\cM_{H,J}|$ called the \emph{geometric realization} (we omit $J$ from the notation). To define this, one uses Pontryagin--Thom collapse and obtains a ``chain complex in spectra''. Then there is a realization construction for such complexes. 

Similarly to classical Floer theory, given a generic homotopy $(H_s, J_s)_{s \in \bR}$ where $\partial_s H_s \leq 0$, $(H_s, J_s)=(H,J)$ for $s\gg 0$ and $(H_s, J_s)=(H',J')$ for $s\ll 0$, one has continuation maps $HF(H;\bS)\to HF(H';\bS)$. 
To define these maps, one assembles the continuation trajectories into a \emph{framed flow bimodule} over $\cM_{H,J}$-$\cM_{H',J'}$, inducing a map of geometric realizations. 

Consider a sequence of Hamiltonians $H_1,H_2,\dots$, such that $H_{i+1}> H_i$ outside a compact subset of $M$ and such that the slope goes to infinity. We obtain a sequence $HF(H_1,\bS)\to HF(H_2,\bS)\to \dots$, and let \emph{the spectral symplectic cohomology} $SH(M,\bS)$ be the homotopy colimit of this sequence. By construction, it is a filtered spectrum. Any other choice of data (such as Hamiltonians $H_i$, almost complex structures, continuation data) gives a homotopy equivalent spectrum, and an equivalent filtered spectrum in the sense of \Cref{definition:equivalence}.

\begin{rk}
    A key contribution of \cite{largethesis}, which builds on \cite{foooexponential}, is smooth gluing. Given an \emph{exact} symplectic manifold with suitably generic Floer data, \cite{largethesis} proves the long-expected but highly non-trivial result that:
    \begin{itemize}
        \item[(i)] moduli spaces of $J$-holomorphic curves (including continuation maps, etc) naturally admit the structure of a smooth manifold with corners, with smooth evaluation maps; c.f. \cite[Sec.\ 6]{largethesis}.
        \item[(ii)] given a stable framing of the ambient manifold, one can coherently frame the above moduli spaces;  c.f. \cite[Sec.\ 7,8]{largethesis}.
    \end{itemize}
In \cite{cotekartalequivariantfloerhomotopy}, the authors further explained how to extend \cite{largethesis} to the setting of coupled moduli spaces arising from the Borel construction (this does not require Morse--Bott gluing). \Cref{subsec:equivariantPSS} and \Cref{subsec:localfloer} of this paper rely on the above results. In addition to \cite{largethesis} (which is publicly available but unpublished), we also refer the reader to \cite[Sec.\ 8]{porcelli2024bordism} which explains many of the key ideas from \cite{largethesis}.
\end{rk}
\subsection{Review of $S^1$-equivariant spectral symplectic cohomology} 
In \cite{cotekartalequivariantfloerhomotopy}, the authors extended the realization construction for flow categories to the Morse--Bott and equivariant settings, and used this to define $S^1$-equivariant versions of $HF(H,\bS)$ and $SH(M,\bS)$ by using what is often referred as the Borel construction. 

\subsubsection{Morse--Bott flow categories}

A \emph{(framed) Morse--Bott flow category} $\cM$ is similar to an ordinary flow category, except that one allows $ob(\cM)$ to be a finite disjoint union of smooth closed manifolds. If $X,Y\subset ob(\cM)$ denote two components and $\cM(X,Y)$ denotes the morphisms from a point on $X$ to one on $Y$. We assume the domain and target maps $\cM(X,Y)\to X,Y$ are smooth, and satisfy a transversality condition; see \cite[Def.\ 2.1]{cotekartalequivariantfloerhomotopy}. The motivating example is the category associated to a Morse--Bott function $f:N\to \bR$: the object space is the union of critical manifolds of $f$ and the morphisms are given by the broken negative gradient trajectories (with respect to a generic metric). The stable framing condition is generalized as a framing of the relative tangent bundle $T_{\cM(X,Y)}-T_X$ twisted by virtual bundles on $X$ and $Y$. See \cite{cotekartalequivariantfloerhomotopy} for more details. 

Given the data of a framed flow category, \cite{cotekartalequivariantfloerhomotopy} explains how to associate a spectrum $|\cM|$, also called the \emph{geometric realization}. If the category and the framings are equivariant with respect to a compact group action, $|\cM|$ is also a genuine equivariant spectrum.

\subsubsection{Data for the Borel construction}\label{subsubsection:data-borel}
Let $M$ be a Liouville manifold and fix a non-degenerate contact form $\lambda_\infty$ on $\partial_\infty M$. Let $a>0$ be a real number which is not equal to the length of a Reeb orbit for $\lambda_\infty$. 

Let $S:= S^\infty= \{(z_1, z_2, \dots) \mid \sum_i |z_i|^2=1,  z_i \in \mathbb{C}\}$. Then $S$ admits a free $S^1$-action with quotient $\pi: S \to \mathbb{CP}^\infty$. There is a standard Morse--Bott function $\tilde{f}: S \to \bR, \tilde{f}(z_1,z_2, \dots)= \sum_i i |z_i|$. It descends to a Morse function on $\mathbb{CP}^\infty$. We consider a generic metric $g$ on $\mathbb{CP}^\infty$ and choose an equivariant lift $\tilde{g}$ to $S^\infty$. The lift determines a connection on the principal $S^1$-bundle $S^\infty\to\mathbb{CP}^\infty$ and we can assume for convenience (by choosing $g$ suitably) that it is flat near the critical sets. Therefore, for every critical set $S^1\cong X\subset S^\infty$, we have a canonical local trivialization $U\times X\subset S^\infty\to \mathbb{CP}^\infty\supset U$.  
 


Let $H: S \times S^1 \times M \to \bR$ be a Hamiltonian having the following properties:
\begin{itemize}
    \item $H$ is invariant under the diagonal $S^1$ action.
    \item $H_z= H_z(-, -): S^1 \times M \to \bR$ has slope $a>0$ for all $z \in S$, and $H_z$ is non-degenerate for all $z \in crit(\tilde{f})$.
    \item in a canonically trivialized neighborhood $U\times X$ of a critical set $H_z$ does not depend on $z\in U$. In other words, $dH$ kills the horizontal directions.
\end{itemize}
Similarly, we let $J: S \times S^1 \to \operatorname{End}(TM)$ be a generic family of almost-complex structures invariant under the diagonal $S^1$ action, cylindrical outside a (uniform) compact subset of $M$, and constant near the critical sets. 

To this data, we associate an $S^1$-equivariant (framed) Morse--Bott flow category $\cM_{\tilde f, H}$ ($J$ is omitted from the notation). The object space is given by pairs $(a,x)$, $a\in crit(\tilde f)$, $x\in orb(H_a)$. They form a disjoint union of $S^1$-torsors. The morphisms are given by the pairs $(\gamma, u)$, where $\gamma:\bR\to S^\infty$ and $u:\bR \times S^1\to M$, satisfying the parameterized Floer equation
\begin{equation}\label{equation:borel-equation}
    \begin{cases}
         \dot{\gamma}+ \nabla_{\tilde{g}} \tilde{f}(\gamma) &= 0 \\
       \partial_s u + J_{\gamma(s)}(\partial_t u - X_{H_{\gamma(s)}}) &=0
    \end{cases}
\end{equation}
asymptotic to $(a_\pm, x_\pm)$. We define $HF_S(H,\bS):=|\cM_{\tilde f, H}|$. 

Similar to before, there are equivariant continuation maps $HF_S(H,\bS)\to HF_S(H',\bS)$, and we define $SH_S(M,\bS)$ to be colimit of $HF_S(H,\bS)$ as the slope of $H$ goes to infinity. We obtain a filtered equivariant spectrum, whose filtration is well-defined up to equivalence.

The construction gives a genuine $S^1$-equivariant spectrum such that $SH_S(M,\bS)\simeq SH(M,\bS)$ as non-equivariant spectra \cite{cotekartalequivariantfloerhomotopy}. However, this genuine equivariant structure is not the one one expects intuitively: in particular the geometric fixed points are trivial for any non-trivial subgroup of $S^1$. We are interested in the underlying Borel equivariant structure on $SH_S(M,\bS)$, while remembering the filtration. 
\begin{exmp}
When $M$ is a point, $SH_S(M,\bS)\simeq \Sigma^\infty S_+$ as genuine equivariant spectra. This follows from \cite[Proposition 4.7]{cotekartalequivariantfloerhomotopy}.
\end{exmp} 
\begin{rk}
Remembering both the filtration and the circle action on $SH_S(M; \bS)$ is crucial for our purposes. Indeed, it is presumably true $SH_S(M \times \mathbb{C}; \bS)\simeq 0$ as a spectrum. \footnote{If $SH(M; \mathbb{S})$ is bounded below as a spectrum, then the claim follows easily from Hurewicz and the equivalence $SH_S(M; \mathbb{S}) \simeq SH(M; \mathbb{S})$ as non-equivariant spectra \cite[Prop.\ 5.12]{cotekartalequivariantfloerhomotopy}. The general statement can presumably be proven using the methods in \cite{seidel2008biased} or \cite{cieliebak2002handle}, but we are not aware of a reference. This claim is not used in this paper.} This implies the same as a Borel equivariant spectrum. On the other hand, if we remember the filtration, but forget about the equivariant structure, we expect the filtration on $SH_S(M \times \mathbb{C})\simeq 0$ to be equivalent to the trivial one. This can presumably be established by the second argument of \cite[(3f)]{seidel2008biased}. In contrast, our main result will imply that $SH_S(M; \bS)$ is never zero as an object of $\Sp^{BS^1}_{fil}$, for any Liouville manifold $M$.
\end{rk}

\section{Approximately autonomous Hamiltonians and the completed Tate cohomology of $SH_S(M,\bS)$}

\subsection{The main theorem and its applications}
The goal of this section is to prove the following theorem:
\begin{thm}\label{thm:tatecohforsh}
Given a complex oriented ring spectrum $R$, 
\begin{equation}
    \tate{R}(SH_S(M,\bS)) \simeq \tate{R}(\Sigma^\infty (M/\partial_\infty M)) \simeq \mathfrak{F}_R^{-1}R^*(M,\partial_\infty M)[[u]]\simeq \mathfrak{F}_R^{-1}R_{2n-*}(M)[[u]]
\end{equation}
where $\Sigma^\infty (M/\partial_\infty M)$ is endowed with the trivial filtration and trivial $S^1$-action. Here, $2n=dim_\bR(M)$.
\end{thm}
Recall that $\mathfrak{F}_R=\{[k]_R(u):k\in\bN\}$ and $\mathfrak{F}_R^{-1}R^*(M,\partial_\infty M)[[u]]$ is $R^*(M,\partial_\infty M)[[u]]$ localized at $[k]_R(u)$, for every $k\geq 1$. 
The second isomorphism follows from \Cref{lem:trivaction} and the third follows from Poincar\'e duality (which holds for $R$ as $TM$ is complex oriented). As we already noted, we could equivalently write the last term as $R_{2n-*}(M)\otimes_{R^*} \tate{R}(\bS)$.

Before moving onto the proof, we discuss some corollaries in the light of examples given in \Cref{sec:filteredeqspec}. 
\begin{cor}\label{cor:eilenbergmaclanetatesh}
We have $\tate{H\bQ}(SH_S(M,\bS))\simeq \tate{H\bZ}(SH_S(M,\bS))\simeq H_{2n-*}(M,\bQ((u)))$. On the other hand  $\tate{H\bF_p}(SH_S(M,\bS))\simeq 0$.
\end{cor}

This follows immediately by applying \Cref{thm:tatecohforsh} to \Cref{exmp:eilenmac}. Note that for a commutative ring $A$, $\tate{HA}(SH_S(M,\bS))$ depends only on the $S^1$-equivariant (co)filtered homotopy type of $SH^*(M,A)$, which does not require Floer homotopy theory to define. For $A=\bQ$, this reaffirms the conclusion of \cite{zhaoperiodic, albers2016symplectic} that the symplectic cohomology with its $S^1$-action and filtration recovers the rational cohomology. For $A=\bZ$ this is a generalization of the computations for a disc and annulus in \cite[Sec.\ 8]{zhaoperiodic}. In this way, we lose the torsion information. See also \cite[Sec.\ 5.1]{albers2016symplectic} for related computations. 

Similarly, by applying \Cref{thm:tatecohforsh} to \Cref{exmp:ordinarymorava} and \Cref{exmp:deformedmorava}, we obtain
\begin{cor}\label{cor:deformedmoravash}
$\tate{K_{p^k}(m)}(SH_S(M,\bS))\simeq K_{p^k}(m)_{2n-*} (M)((u))$.
\end{cor}
Our second main statement is the following theorem, which follows from \Cref{cor:deformedmoravash}:
\begin{thm}\label{thm:hlgrecovery}
For $2(p^m-1)>dim_\bR(M)$ (or for $4(p^m-1)>dim_\bR(M)$ for Weinstein $M$), the groups $\tate{K_{p^k}(m)}(SH_S(M,\bS))$ and $H_{2n-*}(M,\bZ[v_m,v_m^{-1}]/p^k)((u))$ are isomorphic. 
\end{thm}
\begin{proof}
We borrow the idea from \cite[Lemma 5.15]{abouzaidmcleansmithglobal1}. If $2(p^m-1)>dim_\bR(M)$ or $4(p^m-1)>dim_\bR(M)$ and $M$ is Weinstein, $|v_m|=2(p^m-1)$ is larger than the range of the cohomology of $M$. Therefore the Atiyah--Hirzebruch spectral sequence degenerates and $K_{p^k}(m)^* (M,\partial_\infty M)\simeq H^*(M,\partial_\infty M,\bZ[v_m,v_m^{-1}]/p^k)$. The theorem now follows from \Cref{cor:deformedmoravash} and Poincar\'e duality. 
\end{proof}
In particular, $\tate{K_{p}(m)}(SH_S(M,\bS))\cong H_{2n-*}(M,\bF_p[v_m,v_m^{-1}])((u))$, whose rank over $\bF_p[v_m,v_m^{-1}]((u))$ gives the total dimension of $H_{2n-*}(M,\bF_p)$. However, one can recover more information: the isomorphism in \Cref{thm:hlgrecovery} is natural in $k$; hence, the inverse limits of $\tate{K_{p^k}(m)}(SH_S(M,\bS))$ and $H_{2n-*}(M,\bZ[v_m,v_m^{-1}]/p^k)((u))$ in $k$
are also isomorphic. The finitely generated abelian group $H_{2n-*}(M,\bZ)$ can be expressed as a direct sum of cyclic groups (in a non-unique way). If one uses the universal coefficients theorem, and takes a limit, 
\begin{equation}
\lim\limits_{k} H_{2n-*}(M,\bZ[v_m,v_m^{-1}]/p^k)((u))    
\end{equation} 
the $\bZ$ summands in $H_{2n-*}(M,\bZ)$ turn into $\bZ_p[v_m,v_m^{-1}]((u))^\wedge$ and $\bZ/p^l$ summands turn into 
\begin{equation}
    \bZ_p[v_m,v_m^{-1}]((u))^\wedge/p^l\oplus \bZ_p[v_m,v_m^{-1}]((u))^\wedge/p^l
\end{equation}
as modules over $\bZ_p[v_m,v_m^{-1}]((u))^\wedge$ (the $p$-adic completion of $\bZ_p[v_m,v_m^{-1}]((u))$). As a result, one can recover the $p$-torsion part within $H_*(M,\bZ)$, for every prime $p$. Combining this with \Cref{cor:eilenbergmaclanetatesh}, we obtain:
\begin{cor}\label{cor:fullhlgrecovery}
The filtered $S^1$-equivariant homotopy type of $SH_S(M,\bS)$ determines $H_*(M,\bZ)$.
\end{cor}
Notice that we only need the (filtered, $S^1$-equivariant) generalized cohomology of $SH_S(M,\bS)$ with respect to Morava $K$-theories. Therefore less information than $SH_S(M,\bS)$ is actually required. 

Another implication of \Cref{thm:tatecohforsh} is the following (recall that for an abelian group $A$, $A^\wedge:=\lim\limits_{n\in\bN}A/nA$):
\begin{thm}\label{thm:kutatesh}
After simultaneously completing at every prime $p$,  the groups $\tate{KU}(SH_S(M,\bS))$ and $KU_{2n-*} (M)((u))$ become isomorphic. In other words, $\tate{KU}(SH_S(M,\bS))^\wedge\simeq KU_{2n-*} (M)((u))^\wedge$. As a result, the filtered $S^1$-equivariant homotopy type of $SH_S(M,\bS)$ determines $KU_*(M)$ (and hence $KU^*(M)$ by the universal coefficients theorem). 
\end{thm}
This result can be interpreted as saying that the (co)filtered equivariant homotopy type of ``symplectic $K$-homology'' determines the $K$-theory of the underlying manifold. 
\begin{proof}
The isomorphism statement follows from \Cref{thm:tatecohforsh} and \Cref{exmp:ku}. Moreover, as $M$ is of finite type, each $KU_i(M)$ is a finitely generated abelian group. $KU_* (M)((u))$ can be identified with $KU_0(M)((\beta^{-1}u))$ at even degrees and with $KU_1(M)((\beta^{-1}u))$ at odd degrees. Each of these groups split as $(T\oplus F)((\beta^{-1}u))$, where $T$ is the torsion part of $KU_0(M)$, resp. $KU_1(M)$, and $F$ is a finitely generated free abelian group. The completion procedure leaves $T((\beta^{-1}u))$ the same, and $F((\beta^{-1}u))$ turns into a finitely generated free module over $\bZ((\beta^{-1}u))^\wedge$ of the same rank (and the former is still the torsion part). It is clear that $T$ and $F$ are uniquely determined. 
\end{proof}
\begin{rk}\label{rk:treumann}
As we mentioned, \Cref{thm:kutatesh} is closely related to a string theory and mirror symmetry inspired conjecture of Treumann, which states that the Fukaya category determines the complex $K$-theory, see \cite{treumann2019complex}. 
This is clearly false for wrapped Fukaya categories, as there can be subcritical Weinstein manifolds with different complex $K$-theories; however, Treumann conjectures filtered versions of this claim. In other words, the filtered wrapped Fukaya category recovers the complex $K$-theory. \Cref{thm:kutatesh} provides evidence for this conjecture. When the Atiyah--Hirzebruch spectral sequence $H^*(SH_S(M,\bS), KU^*)\Rightarrow KU^*(SH_S(M,\bS))$ degenerates (e.g. when the ordinary symplectic cohomology is supported in even degrees), $KU^*(SH_S(M,\bS))$ is uniquely determined by the symplectic cohomology. Assuming the filtration can be chosen so that the same degeneration claim holds uniformly at every level, $KU^*(SH_S(M,\bS))$ is likely to be determined as a (co)filtered equivariant spectrum by the filtered equivariant symplectic cohomology. It is reasonable to expect that a filtered equivariant version of \cite{sheelthesis,ganatracyclic} holds, i.e. the filtered equivariant symplectic cohomology is determined by the filtered wrapped Fukaya category, and this would prove Treumann's conjecture.	
\end{rk}

\subsection{The strategy for the proof of \Cref{thm:tatecohforsh}}\label{subsection:proof-summary}
Our plan is to exhibit a cofinal sequence of suitably constructed linear Hamiltonians $H: S \times S^1 \times M \to \bR$. For such Hamiltonians, we will show that $HF_S(H,\bS)$ --as an $S^1$-equivariant spectrum-- has the following building blocks:
\begin{itemize}
    \item[(i)] a single copy of $\Sigma^\infty (M/\partial_\infty M)$
    \item[(ii)] Thom spectra of $S^1$-equivariant virtual bundles $\nu \to S^1/C_k$ for various $k$
\end{itemize}
More precisely, $HF_S(H,\bS)$ admits a finite filtration such that the lowest level is equivalent to $\Sigma^\infty (M/\partial_\infty M)$ (with trivial action) and the subsequent subquotients are equivalent to spectra of the form $(S^1/C_k)^\nu$. By \Cref{lem:nontrivaction}, the completed Tate cohomology of the latter vanishes, which implies that $\tate{R}(HF_S(H,\bS)) \simeq \tate{R}(\Sigma^\infty (M/\partial_\infty M))$. From this, we can conclude that $\tate{R}(SH_S(M,\bS)) \simeq \tate{R}(\Sigma^\infty (M/\partial_\infty M))$, as claimed. 
%

Fix a generic large slope $a>0$. One is tempted to use autonomous Hamiltonians $H\to\bR$ of slope $a$ that are small in the interior of $M$. Assuming that equivariant $HF_S(H,\bS)$ is defined, it can be filtered by action, where the constant orbits produce a building block equivalent to $\Sigma^\infty (M/\partial_\infty M)$ and others produce subquotients equivalent to some $(S^1/C_k)^\nu$. Even though our topological framework from \cite{cotekartalequivariantfloerhomotopy} allows the use of such Hamiltonians, to avoid further gluing analysis, we will follow \cite{zhaoperiodic} and use approximately autonomous Hamiltonians. In other words, we perturb autonomous Hamiltonians to non-autonomous ones near the non-constant orbits. The finite filtration we put on the corresponding spectra is a coarse version of the action filtration.

\subsection{Constructing approximately autonomous Hamiltonians}

\subsubsection{Perturbing an autonomous Hamiltonian}\label{subsubsection:perturbing-autonomous}
Fix a slope $a>0$ that is different from the length of any Reeb orbit of $\lambda_\infty$. Following \cite[\S5]{zhaoperiodic}, it will be convenient to fix a particular class of Hamiltonians of slope $a$ whose dynamics are simple to analyze. The following lemma is essentially proved in \cite[\S5]{zhaoperiodic}. 
\begin{lem} \label{lemma:orbits-characterization}
Given any $0<\epsilon_a \ll 1$, there exists a Hamiltonian $\overline{H}^a: M \to \bR$ with the following properties: 
\begin{equation}
	\overline{H}^a(x)=\begin{cases}
		-F(x),& \text{for }x\in int(\overline{M})\\
		(r-1)^2/2	,&	x=(r,y)\in [1+\epsilon_a,a+1-\epsilon_a]\times\partial\overline{M}\\
		 a(r-1)-a^2/2,& x=(r,y)\in [a+1,\infty)\times\partial\overline{M},
	\end{cases}
\end{equation}
where $F:  \overline{M} \to \bR_{\leq 0}$ is Morse,\footnote{Note that here we differ slightly from \cite[\S5]{zhaoperiodic}, whose Hamiltonian is non-positive in $\overline{M}$. We opted for a slightly different construction to simplify the proof of \Cref{lem:nonequivariant-pss}; see \Cref{remark:h-positive-reason}.} bounded of size $\epsilon_a$ (up to a universal constant) in the $C^2$ norm, and its critical points are contained in the interior. We may assume that the time-$1$ orbits of $\overline{H}^a$ all have different actions, and are of the following two types:
\begin{itemize}
    \item[(a)] critical points of $F$, all of which are non-degenerate (see \cite[Prop.\ 6.1.5]{audindamianmorse})
    \item[(b)] of the form $\{\ell+1 \}\times \gamma_0$, where $\ell\in (0,a)$ and $\gamma_0$ is a length $\ell$ Reeb orbit. Such orbits are transversally non-degenerate and come in $S^1$-families. Their action is $p_\ell:= \ell^2/2+\ell$. 
\qed
\end{itemize}
\end{lem}

We now introduce time-dependent perturbations of  $\overline{H}^a$ near each critical manifold of type (b). Let $h_0: S^1\to [0,1]$ be a Morse function with a maximum at $1/2$ and minimum at $0$. For a non-constant $k$-fold orbit $\gamma$ of $\overline{H}^a$, define a function $h_t$ on $Im(\gamma)$ by $h_t(\gamma(t'))=h_0(k(t'-t))$ and extend $h_t$ to a small tubular neighborhood $V_\gamma$ of $Im(\gamma)$. Consider $\overline{H}^a_{\epsilon, \gamma}=\overline{H}^a+\epsilon h_t$, for a small $0< \epsilon \ll \epsilon_a$. It is shown in \cite[Prop.\ 2.2]{cieliebakhoferwysockiapplsymplchlg} that 
\begin{itemize}
	\item the Hamiltonian $\overline{H}^a_{\epsilon, \gamma}$ has two non-degenerate orbits within $V_\gamma$ given by $\gamma^-(t)=\gamma(t)$ and $\gamma^+(t)=\gamma(t+1/(2k))$ 
	\item $\cA_{\overline H^a_\epsilon}(\gamma^-)=\cA_{\overline H^a}(\gamma)$, $\cA_{\overline H^a_\epsilon}(\gamma^+)=\cA_{\overline H^a}(\gamma)-\epsilon$, and $ind(\gamma^-)=ind(\gamma^+)+1$
	\item there are exactly two Floer trajectories from $\gamma^-$ to $\gamma^+$ within $V_\gamma$ 
\end{itemize}
We make this perturbation for every non-constant orbit, and denote the resulting Hamiltonian by $\overline{H}^a_\epsilon: S^1 \times M \to \bR$. More precisely, as non-constant orbits come in circle families, we make a choice of representative $\gamma:S^1\to M$ for each of them. Then we apply the procedure above for each chosen orbit parametrization. For small enough $\epsilon$, the perturbation will not create new $1$-periodic orbits other than $\gamma^\pm$. In other words:
\begin{lem}\label{lemma:perturbed-orbits-characterization}
The time-$1$ orbits of $\overline{H}^a_\epsilon$ are of the following two types:
\begin{itemize}
    \item[(a)] the critical points of $F$
    \item[(b)] a pair of non-stationary orbits $\gamma^+$ and $\gamma^{-}$ for every Reeb orbit $\gamma$. If $\gamma$ has length $\ell$, then $\gamma^\pm$ have action arbitrarily close to $ p_\ell:= \ell^2/2 +\ell$. 
\qed
\end{itemize} 
\end{lem}
In particular, the Hamiltonian $\overline{H}^a_\epsilon$ is non-degenerate and it has only finitely many orbits.

Choice of a generic almost complex structure $S^1\to M$ allows one to define a flow category $\cM_{\overline{H}^a_\epsilon,J}$ and its geometric realization $HF(\overline H^a_\epsilon,\bS)$. There is a finite, increasing action filtration on $\cM_{\overline H^a_\epsilon,J}$ and $HF(\overline H^a_\epsilon,\bS)$. More precisely, let $F^p\cM_{\overline H^a_\epsilon,J}$ denote the subcategory of $\cM_{\overline H^a_\epsilon,J}$ spanned by orbits of action at most $p$, and let $F^pHF(\overline H^a_\epsilon,\bS)$ denote its geometric realization. For $p\leq p'$, $F^p\cM_{\overline H^a_\epsilon,J}\hookrightarrow F^{p'}\cM_{\overline H^a_\epsilon,J}$ is an inclusion of framed flow categories in the sense of \cite{cotekartalequivariantfloerhomotopy}, and it induces a map $F^pHF(\overline H^a_\epsilon,\bS)\to F^{p'}HF(\overline H^a_\epsilon,\bS)$ of geometric realizations. We denote the homotopy cofiber of this map by $F^{(p,p']}HF(\overline H^a_\epsilon,\bS)$. It coincides with the geometric realization of the quotient flow category by \cite[Cor.\ 2.36]{cotekartalequivariantfloerhomotopy}.\footnote{Consider an appropriate fully faithful inclusion $\iota:\cM_1\hookrightarrow\cM_2$ of flow categories such that any morphism whose domain is in the image of $\iota$ is in the image of $\iota$ (in particular, so is its target). The quotient flow category $\cM_2/\cM_1$ is defined to be the flow category with objects $Ob(\cM_2)\setminus Ob(\cM_1)$ and all the morphisms of $\cM_2$ between them. Roughly, the inclusion and quotient can be thought as analogues of the truncations $\tau_{\geq 0}$ and $\tau_{< 0}$ of chain complexes. See \cite[Def.\ 2.33]{cotekartalequivariantfloerhomotopy}}

If $p$ is the action of $\gamma$, there exists a $\delta>0$ such that $\gamma^\pm$ have action within the window $(p-\delta,p+\delta)$. We assume $\epsilon$ is small enough, so that there exists a $\delta>0$ that works for every orbit and the action windows $(p-\delta,p+\delta)$ corresponding to different orbits do not intersect. We will show the following non-equivariant statements
\begin{enumerate}
	\item $F^0 HF(\overline H^a_\epsilon,\bS)\simeq \Sigma^\infty (M/\partial_\infty M)$ and the map $\Sigma^\infty (M/\partial_\infty M)\to HF(\overline H^a_\epsilon,\bS)$ is compatible with the continuation maps $HF(\overline H^a_\epsilon,\bS)\to HF(\overline H^{a'}_\epsilon,\bS)$
	\item\label{item:localfloer} if $p$ is the action of a non-constant $k$-fold orbit $\gamma$ of $\overline{H}^a$, then the spectrum $F^{(p-\delta,p+\delta]} HF(\overline H^a_\epsilon,\bS)$ 
	is equivalent to $Im(\gamma)^\nu\simeq (S^1/C_k)^\nu$, for a virtual vector bundle $\nu$ on $Im(\gamma)$.
\end{enumerate}
We will also establish equivariant versions of these two statements. For this, we first need to produce an equivariant version of $HF(\overline H^a_\epsilon,\bS)$ that admits a \emph{coarse action filtration} that roughly matches the filtration on $HF(\overline H^a_\epsilon,\bS)$. More precisely, we shall construct approximately autonomous Hamiltonians 
\begin{equation}\label{equation:equivariant-hams}
    H^a_\epsilon:S\times S^1\times M\to\bR
\end{equation} satisfying the required conditions recalled in \S3.4.2 (see \cite{cotekartalequivariantfloerhomotopy} for a more detailed discussion), and a filtration on $HF_S(H^a_\epsilon,\bS)$ such that 
\begin{enumerate}[label=(\theenumi\textquotesingle)]
	\item\label{item:f0eq} $F^0 HF_S(H^a_\epsilon,\bS) \simeq \Sigma^\infty (S\times M/S\times \partial_\infty M)\simeq \Sigma^\infty ( M/ \partial_\infty M)$ as $S^1$-equivariant spectra and the map $\Sigma^\infty  ( M/ \partial_\infty M)\to HF_S(H^a_\epsilon,\bS)$ is compatible with the continuation maps $HF_S(H^a_\epsilon,\bS)\to HF_S(H^{a'}_\epsilon,\bS)$
	\item\label{item:localeq} if $p$ is the action of a non-constant $k$-fold orbit $\gamma$ of $\overline{H}^a$, then the spectrum $F^{(p-\delta,p+\delta]} HF_S(H^a_\epsilon,\bS)$ 
	is $S^1$-equivariantly equivalent to $Im(\gamma)^\nu\simeq (S^1/C_k)^\nu$, for a virtual equivariant vector bundle $\nu$ on $Im(\gamma)$
\end{enumerate}
We will prove \ref{item:f0eq} in \Cref{subsec:equivariantPSS} and \ref{item:localeq} in \Cref{subsec:localfloer}. Assuming these, we can complete the proof of \Cref{thm:tatecohforsh}:
\begin{proof}[Proof of \Cref{thm:tatecohforsh}]
Consider a sequence of real numbers $0<a_1< a_2< \dots$ with $a_i \to \infty$ and each $a_i$ is different from the lengths of non-constant Reeb orbits. Associated to this, we have a sequence $H^{a_k}_\epsilon$ of approximately autonomous Hamiltonians\footnote{The subscript $\epsilon$ is an abuse of notation, as each of $H^{a_k}_\epsilon$ requires a sequence of small numbers, and these sequences can be different from each other; see \Cref{subsubsection:borel-jingyu-data}.} and we choose equivariant continuation data $H^{a_1}_{\epsilon} \rightsquigarrow H^{a_2}_{\epsilon} \rightsquigarrow \dots$ 

We can now use the filtration $HF_S(H^{a_1}_{\epsilon},\bS)\to HF_S(H^{a_2}_{\epsilon},\bS)\to\dots$ of $ SH_S(M,\bS)$ to compute the completed Tate cohomology of $SH_S(M,\bS)$. This filtration is equivalent to 
\begin{equation}
	\Sigma^\infty (M/\partial_\infty M)\to HF_S(H^{a_1}_{\epsilon},\bS)\to HF_S(H^{a_2}_{\epsilon},\bS)\to\dots 
\end{equation}
where the composition $\Sigma^\infty (M/\partial_\infty M)\to HF_S(H^{a_k}_{\epsilon},\bS)$ is equivalent to the PSS map itself by the compatibility statement in \ref{item:f0eq}. Therefore, to conclude the proof, we only need to show that this map induces an equivalence 
\begin{equation}
	\tate{R}(HF_S(H^{a_k}_{\epsilon},\bS))\to \tate{R}(\Sigma^\infty (M/\partial_\infty M))
\end{equation}
for all $k$. We prove $\tate{R}(F^pHF_S(H^{a_k}_{\epsilon},\bS))\to \tate{R}(\Sigma^\infty (M/\partial_\infty M))$ is an equivalence for each $p$, which completes the proof as $F^pHF_S(H^{a_k}_{\epsilon},\bS))=HF_S(H^{a_k}_{\epsilon},\bS))$ for $p\gg 0$. For $p=0$, this follows from \ref{item:f0eq}. As one increases $p$, as long as one does not cross a $\delta$ neighborhood of the action on an orbit of $\overline{H}^{a_k}$, $F^pHF_S(H^{a_k}_{\epsilon})$ does not change. On the other hand, by \ref{item:localeq}, the homotopy fiber of 
\begin{equation}
\tate{R}	(F^{p+\delta}HF_S(H^{a_k}_{\epsilon}))\to \tate{R}	(F^{p-\delta}HF_S(H^{a_k}_{\epsilon}))
\end{equation}
is given by $\tate{R}((S^1/C_m)^\nu)$ for some $m\geq 1$ and for some $S^1$-equivariant bundle $\nu$ on $S^1/C_m$. We have shown $\tate{R}((S^1/C_m)^\nu)\simeq 0$ in \Cref{lem:nontrivaction}. As a result $\tate{R}(F^{p-\delta}HF_S(H^{a_k}_{\epsilon}))$ and $\tate{R}(F^{p+\delta}HF_S(H^{a_k}_{\epsilon}))$ are the same for $p$ in the action spectrum. Therefore, 
\begin{equation}
	\tate{R}(F^{p}HF_S(H^{a_k}_{\epsilon}))\simeq \tate{R}(F^{0}HF_S(H^{a_k}_{\epsilon}))\simeq \tate{R}(\Sigma^\infty (M/\partial_\infty M))
\end{equation}
for $p\gg 0$. This finishes the proof. 
\end{proof}

\subsubsection{Approximately autonomous Hamiltonians on the Borel construction and coarse action filtrations}\label{subsubsection:borel-jingyu-data}
Here we construct the promised Hamiltonians arising in \eqref{equation:equivariant-hams}. Recall that Floer trajectories do not need to decrease the action; however, we will perturb autonomous Hamiltonians in such a way that Floer trajectories are ``almost action decreasing'', guaranteeing us a filtration that approximates the action filtration on orbits in $M$.

Fix a slope $a>0$ and let $\overline H^a:M\to \bR$ be as before. Assume without loss of generality that the actions of non-constant orbits are positive and different from each other. As a result, there exists $\delta>0$ such that $\delta$ is smaller than the action difference between any two non-constant orbits, as well the action difference between a constant and a non-constant orbit. 

To begin with, we let 
\begin{equation}
    H^a: S \times S^1 \times M \to \bR
\end{equation} be the pullback of $\overline{H}^a$ under the map $S \times S^1 \times M \to S^1 \times M$ which forgets the first component.

We would like to input $H^a$ into the Borel construction from \Cref{subsubsection:data-borel}. However, this is not allowed, because $H^a_z= H(z, -, -)$ fails to be non-degenerate over all $z \in crit(\tilde{f})$.

We will therefore deform $H^a$ by a function $h:S\times S^1\times M\to \bR$ supported on a tubular neighborhood of the non-constant orbits. More precisely, for a non-constant $k$-fold orbit $\gamma$ of $\overline H^a$, let $V_\gamma$ denote a tubular neighborhood, and let $b_0,b_1,b_2,\dots \in S$ be a choice of a critical point from each component of $crit (\tilde f)$. Consider a function $h:S\times S^1\times M\to [0,1/2]$ such that

\begin{enumerate}
	\item\label{item:firstitem} $h$ is equivariant with respect to the diagonal action of $S^1$
	\item $h|_{\{b_i\}\times S^1\times Im(\gamma)}: S^1\times Im(\gamma)\to\bR$ is equal to $\epsilon_ih_t$ where $h_t$ is as before and $\epsilon_i>0$
	\item\label{item:thirditem} $h|_{\{b_i\}\times S^1\times V_\gamma}$ is obtained by extending $h|_{\{b_i\}\times S^1\times Im(\gamma)}$ as in \cite{cieliebakhoferwysockiapplsymplchlg} and $h|_{\{b_i\}\times S^1\times M}$ is supported in the union of $\{b_i\}\times S^1\times V_\gamma$
	\item\label{item:locconstant} $h$ is constant in the horizontal directions near the $S^1$-orbit of $b_i$ (see \Cref{subsubsection:data-borel})
\end{enumerate}
The conditions so far ensure that the restriction of $H^a+h$ to $\{b_i\}\times S^1\times M$ is a perturbed Hamiltonian of the form considered in \Cref{subsubsection:perturbing-autonomous}; in particular, it has a constant orbit corresponding to each critical point of $\overline H^a$ and two non-constant orbits corresponding to every non-constant orbit of $\overline H^a$. We also assume
\begin{enumerate}[resume]
	\item the numbers $\epsilon_i$ are sufficiently small so that, for any $i,j$ and pair of (possibly constant) orbits $\gamma,\gamma'$ of $\overline{H}^a$ of different action, the action gap between an orbit of $(H^a+h)|_{\{b_i\}\times S^1\times M}$ obtained by perturbing $\gamma$ and between an orbit of $(H^a+h)|_{\{b_j\}\times S^1\times M}$ obtained by perturbing $\gamma'$ is still strictly more than $\delta$
\end{enumerate}
In other words, the generators of $\cM_{\tilde{f}, H^a+h,J}$ obtained by perturbing different non-constant orbits of $\overline{H}^a$ still have an action gap of $\delta$. Moreover, the action of every such orbit is more than $\delta$. We can filter the objects of $\cM_{\tilde{f}, H^a+h,J}$ and to ensure the morphisms decrease the action (up to a small error), we impose
\begin{enumerate}[resume]
    \item\label{item:integralbound} along any gradient trajectory of $\tilde f$, the integral $\int_{\bR}\sup_{S^1\times M} |dh(\nabla_{\tilde{g}}\tilde f)|ds $ is strictly less than $\delta/2$
\end{enumerate}
Note that $\sup_{S^1\times M} |dh(\nabla_{\tilde{g}}\tilde f)|$ is compactly supported because of \eqref{item:locconstant}. More precisely, $h$ is constant along horizontal directions and $\nabla_{\tilde{g}}\tilde f$ is horizontal. 
\begin{lem}
Functions $h:S\times S^1\times M\to [0,1/2]$ satisfying \eqref{item:firstitem}-\eqref{item:integralbound} exist.
\end{lem}
\begin{proof}
One can construct such an $h$ over each $S^{2i+1}\subset S$, and can guarantee \eqref{item:integralbound} by rescaling it by a small constant. The main difficulty is to guarantee it over the entire $S$. 

Let $U_i\subset S$ denote a small $S^1$-equivariant neighborhood of the orbit of $b_i$, on which the connection is flat, and let $S_i$ denote a horizontal slice along $b_i$ such that $U_i$ is the union of orbits of elements of $S_i$. Construct $h$ over $b_i$, extend $S^1$-equivariantly to its orbit and extend horizontally to $U_i$. This is the same as constructing a function on $\{b_i\}\times S^1\times M\to \bR$, and extending to $h_{pre}:S_i\times S^1\times M\to \bR$ constantly in $S_i$ direction, and then extending $S^1$-equivariantly to $h_{pre}:U_i\times S^1\times M\to \bR$. Observe that $dh_{pre}(\nabla_{\tilde{g}}\tilde f)=0$. Indeed, $\nabla_{\tilde{g}}\tilde f$ is tangent to any horizontal slice $zS_i\subset U_i$ ($z\in S^1)$, and $h_{pre}$ is constant in $zS_i$ direction by equivariance. 

We obtain a function over $\bigcup U_i\times S^1\times M$ satisfying \eqref{item:firstitem}-\eqref{item:locconstant}. Using an equivariant cutoff function $\rho:S\to\bR$ that has arbitrarily large support in $\bigcup U_i$, extend $h_{pre}$ to all of $S$ (i.e. $h=\rho h_{pre}$). We only have to check \eqref{item:integralbound} holds. 
The function $\sup_{S^1\times M} |dh(\nabla_{\tilde{g}}\tilde f)|$ is supported near the boundary of $U_i$, as $h$ and $h_{pre}$ agree over the set $\{\rho=1\}\subset U_i$ and $dh_{pre}(\nabla_{\tilde{g}}\tilde f)=0$. Moreover, 
\begin{equation}
   dh(\nabla_{\tilde{g}}\tilde f)=\rho dh_{pre}(\nabla_{\tilde{g}}\tilde f)+h_{pre} d\rho(\nabla_{\tilde{g}}\tilde f) =h_{pre} d\rho(\nabla_{\tilde{g}}\tilde f) 
\end{equation}
Therefore, $\sup_{S^1\times M} |dh(\nabla_{\tilde{g}}\tilde f)|= \sup_{S^1\times M} |h_{pre} d\rho(\nabla_{\tilde{g}}\tilde f)|\leq \epsilon_i |d\rho(\nabla_{\tilde{g}}\tilde f)|$ over the points of $U_i$ (the right hand side does not depend on $S^1\times M$ component). The integral of $|d\rho(\nabla_{\tilde{g}}\tilde f)|$ along a partial negative gradient trajectory of $\tilde f$ crossing $U_i$ is bounded above by $2$ (one can split this integral into two intervals, where $d\rho(\nabla_{\tilde{g}}\tilde f)$ is positive and negative and apply fundamental theorem of calculus. As a result, the integral of $\sup_{S^1\times M} |dh(\nabla_{\tilde{g}}\tilde f)|$ along a partial negative gradient trajectory of $\tilde f$ crossing the support near $\partial U_i$ is bounded above by $2\epsilon_i$. \eqref{item:integralbound} holds under the assumption that $2\sum\epsilon_i<\delta/2$. See \cite[Lem.\ 5.4]{zhaoperiodic} for a similar computation.
\end{proof} 
We denote $H^a+h$ also by $H^a_\epsilon$ to stress its dependency on $\{\epsilon_i\}$. The purpose of conditions \eqref{item:firstitem}-\eqref{item:integralbound} is to guarantee: 
\begin{lem}
The morphisms of $\cM_{\tilde{f}, H^a_\epsilon,J}$ between objects obtained by perturbing different orbits of $\overline{H}^a$ decrease the action. More generally, a morphism can increase the action by at most $\delta/2$.
\end{lem}
\begin{proof}
The first claim follows from the second: two such objects have an action gap of at least $\delta$. Therefore, a morphism between them cannot increase the action. 

To prove the second claim, consider a morphism $(\eta,u)$ of $\cM_{\tilde{f}, H^a_\epsilon,J}=\cM_{\tilde{f}, H^a+h,J}$. Here $\eta$ is a negative gradient trajectory of $\tilde f$ and $u$ satisfies 
\begin{equation}
	\partial_s u + J_{\eta(s)}(\partial_t u - X_{H^a+h_{\eta(s)}})=0.
\end{equation}
Let $H_s=H^a+h_{\eta(s)}$. The energy of such a cylinder is given by 
\begin{equation}\label{eq:actionintegralformula}
	E(u)=\cA(x_{inp})-\cA(x_{out})+\int_{\bR \times S^1}\partial_s H_s (u(s,t))dsdt.
\end{equation} 

Observe that $\partial_s H_s (u(s,t))=-dh(\nabla_{\tilde{g}}\tilde f)$ (because $H^a$ is independent of $s$). 
Therefore, 
\begin{equation}
\bigg|	\int_{\bR \times S^1}\partial_s H_s (u(s,t))dsdt\bigg|\leq 	\int_{\bR \times S^1}|dh(\nabla_{\tilde{g}}\tilde f)|dsdt\leq \int_{\bR}\sup_{S^1\times M}|dh(\nabla_{\tilde{g}}\tilde f)|ds\leq \delta/2.
\end{equation}

\end{proof}
As a result, given $p'$ that is at least $\delta$ away from any action of any orbit of $\overline{H}^a$, the span of objects of action at most $p'$ defines a flow subcategory $F^{p'}\cM_{\tilde{f}, H^a_\epsilon,J}$ (more precisely, there is an inclusion $F^{p'}\cM_{\tilde{f}, H^a_\epsilon,J} \hookrightarrow\cM_{\tilde{f}, H^a_\epsilon,J}$ in the sense of \cite{cotekartalequivariantfloerhomotopy}). This may not be the case for $p'$ closer to the action spectrum of $\overline H^a$. Clearly, varying $p'$ continuously without getting closer to the action spectrum than $\delta$ does not change the subcategory or its realization. Therefore, we will be specifically interested in $F^{p\pm\delta}\cM_{\tilde{f}, H^a_\epsilon,J}$, for $p$ in the action spectrum of $\overline{H}^a$, their geometric realizations, and the homotopy cofiber
\begin{equation}
F^{(p-\delta,p+\delta]}HF_S(H^a_\epsilon):=cofib(|F^{p-\delta}\cM_{\tilde{f}, H^a_\epsilon,J}|\to |F^{p+\delta}\cM_{\tilde{f}, H^a_\epsilon,J}|).
\end{equation}

\section{An equivariant PSS map}\label{subsec:equivariantPSS}
The goal of this section is to prove \Cref{item:f0eq} from \Cref{subsubsection:perturbing-autonomous}. This follows from the following:

\begin{prop}\label{proposition:equivariant-pss}  
    There is an \emph{$S^1$-equivariant PSS map} 
	\begin{equation}
     \Sigma^\infty (M/\partial_\infty M) \simeq \Sigma^\infty (S\times M/S\times \partial_\infty M) \to HF_S(H; \bS)
	\end{equation} 
satisfying the following properties:
	\begin{itemize}
		\item Given $H \rightsquigarrow H'$, we have a homotopy-commutative diagram
\begin{equation}
	\begin{tikzcd}
		\Sigma^\infty (M/\partial_\infty M) \ar[r] \ar[dr] & HF_S(H; \bS) \ar[d] \\
		& HF_S(H'; \bS)
	\end{tikzcd}
\end{equation}
		\item when $H=H^a_\epsilon$, it factors through an equivalence $ \Sigma^\infty (M/\partial_\infty M)\xrightarrow{\simeq} F^0HF_S(H^a_\epsilon; \bS)$.		
	\end{itemize}
\end{prop}
(Here $HF_S(H; \bS)$ is defined as in \Cref{subsubsection:data-borel}, and depends implicitly on a choice of almost-complex structure. The notation $H \rightsquigarrow H'$ just indicates a homotopy of Floer data.)

As a warm-up for the equivariant case, we prove the following non-equivariant version first
\begin{lem}\label{lem:nonequivariant-pss}
	There is a \emph{PSS map} 
	\begin{equation}
		\Sigma^\infty (M/\partial_\infty M)  \to HF(H; \bS)
	\end{equation}
	satisfying the following properties:
	\begin{itemize}
		\item Given $H \rightsquigarrow H'$, we have a homotopy-commutative diagram
		\begin{equation}\label{equation:non-equivariant-pss-commute}
			\begin{tikzcd}
				\Sigma^\infty (M/\partial_\infty M)
				\ar[r] \ar[dr] & HF(H; \bS) \ar[d] \\
				& HF(H'; \bS)
			\end{tikzcd}
		\end{equation}
		\item when $H=\overline H^a_\epsilon$, it factors through an equivalence $\Sigma^\infty (M/\partial_\infty M) \xrightarrow{\simeq} F^0HF(\overline H^a_\epsilon; \bS)$ 
	\end{itemize}
\end{lem}
(Here $HF(H; \bS)$ is defined as in \Cref{subsection:spectral-sh}, and depends implicitly on a choice of almost-complex structure.)

To construct the non-equivariant map, fix a function $f$ that is Morse, proper, with critical points only in the interior, and bounded above (also fix a Morse--Smale metric for $f$). The same argument as in \cite[Thm.\ 3.11]{cotekartalequivariantfloerhomotopy} shows that $|\cM_{-f}|\simeq \Sigma^\infty M_+$, as $-f$ is bounded below; \emph{mutatis mutandis} the argument also shows that $|\cM_{f}|\simeq \Sigma^\infty (M/\partial_\infty M)$. (Later on, it will be convenient to choose $f=-\overline{H}^a$, which matches $F$ in the interior, but this is not necessary yet.)

Let $(H, J)$ be Floer data on $M$ satisfying the usual properties reviewed in \Cref{subsection:spectral-sh}.   That is, $H$ is a map from $S^1$ to the space of linear at infinity non-degenerate Hamiltonians on $M$, and $J$ is a map from $S^1$ to the space $\mathcal{J}(M)$ of cylindrical compatible almost-complex structures on $M$. Let $D = \mathbb{CP}^1\setminus \{0\}$ and fix coordinates
\begin{equation}\label{equation:ends-disk}
\mathbb{R} \times S^1 \simeq D- \{\infty\}, (s, t) \mapsto e^{s+it}
\end{equation}
We now introduce Floer data on $M$ parametrized by $D$, which we shall temporarily denote by $(H_{ext}, J_{ext})$. Here $H_{ext}$ is a $1$-form on $D$ valued in the space of smooth functions on $M$ (which are linear at infinity), while $J_{ext}$ is a map from $D$ to $\mathcal{J}(M)$. For $s \gg 1$, we require that $(H_{ext}, J_{ext})_{s,t}= (H_t dt, J_t)$. Otherwise, the only condition which we impose on $(H_{ext}, J_{ext})$ is that $J_{ext}$ is sufficiently generic to ensure transversality of the moduli spaces and evaluation maps which we will shortly introduce.  (Later on, it will be useful to impose additional conditions on $H$, but these are not needed yet.)

By abuse of notation, we now drop the subscripts ``${ext}$'' and instead just denote the Floer data parametrized by $D$ by $(H, J)$.  (To avoid notational misunderstandings: we started with Floer data $(H, J)$, then we used this data to define a \emph{family} of Floer data $(H_{ext}, J_{ext})$. But now, we will stop writing the subscript ``${ext}$'' and just write $(H, J)$. In other words, the new $(H,J)$ is an extension of the old one.) 

We now define a bimodule $\cN^{PSS}$ over $\cM_{f}$-$\cM_{H,J}$ by associating to $(x,y)\in ob(\cM_f) \times ob(\cM_{H, J})$ the usual compactified moduli space of ``spiked discs''. More precisely, this is the compactified moduli space of pairs $(\theta, u)$ such that:
\begin{itemize}
\item $\theta: (-\infty, 0] \to M$ is a negative gradient (half-)trajectory for the above Morse data
\item $u: D \to M$ satisfies the Cauchy–Riemann equation $$(du - X_H)^{0,1}=0$$ with respect to the chosen $D$-parametrized Floer data above. 
\item $\theta$ is asymptotic to $x$, $u$ is asymptotic to $y$ and $\theta(0)=u(\infty)$.
\end{itemize}



The construction of smooth charts on these moduli spaces can be done using the methods of \cite[Sec.\ 6]{largethesis}. More precisely, we are dealing with a fiber product of two moduli spaces, both of which have the structure of a smooth manifold with corners with smooth evaluation maps. (For the Morse moduli space, this is by \cite[Thm.\ 2.3, Rmk.\ 2.4]{wehrheim2012smooth}; for the Floer moduli space, this follows from the methods of \cite[Sec.\ 6]{largethesis}.) Transversality of the evaluation maps is a property which can be checked on pairs of strata (see \cite[Def.\ 16, Thm.\ 4]{hajek2014manifolds}). It thus reduces to the well-known fact (cf.\ \cite[Sec.\ 3.1]{abouzaid2011cotangent}) that the Floer evaluation map can be made transverse to any submanifold for generic choice of $J$.

To define a map $|\cM_{f}|\to |\cM_{ H,J}|$, we need to frame this bimodule. For this purpose, it is more natural to use the description of the framings on $\cM_{H,J}$ in terms of negative caps (see \cite[Lem.\ 5.5]{cotekartalequivariantfloerhomotopy} and the discussion preceding it). If $W_x$ denotes the framing bundle over $\{x\}\subset crit(f)$ and $V_y^-$ is the bundle over $\{y\}\subset orb(H)$ defined as in loc.\ cit.\, then a point in the moduli of spiked discs has tangent space given by 
\begin{equation}\label{equation:framing-pss}
    W_x+V_y^--\bR^{2k}-T_M,
\end{equation} where $k$ is the stabilization constant (i.e. we fix $T_M\oplus \bC^k\simeq \bC^{n+k}$). This follows from the description of the moduli space of spiked disks as the fiber product of the moduli of half gradient trajectories from $x$ and the moduli of half discs to $x$ over their evaluation maps to $M$. The former has tangent space that can be identified with $W_x$ and the tangent space of the discs can be identified with $V_y^--\bC^k$. By subtracting $T_M$ pulled back along the evaluation map at the spike, we obtain the asserted equivalence. Combining this with \cite[Lem.\ 5.5]{cotekartalequivariantfloerhomotopy} gives us the framings. Therefore, we have a map \begin{equation}\label{equation:pss-last}
	\Sigma^\infty (M/\partial_\infty M)=|\cM_{f}|\to |\cM_{ H,J}|=HF(H,\bS).
\end{equation} 

The verification that \eqref{equation:pss-last} does not depend on $f$ and $(H,J)$ up to homotopy is standard. More precisely, in the formalism of \cite{cotekartalequivariantfloerhomotopy}, moduli spaces of continuation trajectories associated to an interpolation of Morse (resp.\ Floer) data are organized into a flow bimodule \cite[Sec.\ 3.1]{cotekartalequivariantfloerhomotopy}. These bimodules induce maps on geometric realisations, and homotopies of bimodules give rise homotopic maps, etc. One can therefore repeat the same formal arguments as in ordinary Floer homology theory. Analogous considerations also imply the commutativity of \eqref{equation:non-equivariant-pss-commute}.

We now turn to establishing the second bullet point in \Cref{lem:nonequivariant-pss}.

To this end, it is now convenient to set $f= -\overline{H}_a$, which matches $F$ in the interior. It is also convenient to impose some additional conditions on the Hamiltonian-valued $1$-form $H$ on $D$. Referring to the coordinates \eqref{equation:ends-disk}, let us write $H_s(-)=H(s,-)$. Then we require that
\begin{itemize}
\item for $s\leq 0$, $H_s = \overline{H}_a^{\epsilon}  dt$
\item for $0 \leq s \leq 1$, $H_s$ interpolates between $\overline{H}_a^{\epsilon}$ and $\overline{H}_a$. (More precisely, let $\chi: \mathbb{R} \to \mathbb{R}$ be a non-increasing function which is $1$ for $s \leq 0$ and $0$ for $s \geq 1$. Then set $H_s=\overline{H}_a^\epsilon \chi  + \overline{H}_a (1- \chi)$.)
\item for $1 \leq s \leq 2$, $H_s$ interpolates between $\overline{H}_a$ and $e^{-2s} \overline{H}_a$.
\item for $s \geq 2, H_s = e^{-2s} \overline{H}_a dt$
\end{itemize}
Note that this construction ensures that $\partial_s H_s \leq 0$.
We leave to the reader to check that $H$ extends (uniquely) to a Hamiltonian-valued $1$-form on $D$, due to the decay condition as $s \to \infty$.

The virtue of this choice of Hamiltonian is the following: if $y \in orb(\overline{H}^a_\epsilon)$ denotes the output of the spiked disk, then the energy identity \eqref{eq:actionintegralformula}, implies that $-\cA(y) \geq 0$. Since, by \Cref{lemma:perturbed-orbits-characterization}, the non-constant orbits have action roughly $\ell^2/2+ \ell >0$, it follows that $y$ cannot be a non-constant orbit.  Hence our bimodule $\cN^{PSS}$ factors through $F^0\cM_{\overline H^a_\epsilon,J}$ (and the induced map factors through $F^0 HF(\overline H^a_\epsilon,\bS)$).
\begin{rk}\label{remark:h-positive-reason}
    This argument is the reason why we required $\overline{H}^a$ in \Cref{lemma:orbits-characterization} to be non-negative, since otherwise we could not ensure $\partial_s H_s \leq 0$.
\end{rk}

Having shown this, the desired equivalence $\Sigma^\infty (M/\partial_\infty M)\to F^0 HF(\overline H^a_\epsilon,\bS)$ now reduces to classical Floer theory. For example, one can further filter $\cM_{f}$ and $F^0\cM_{\overline H^a_\epsilon,J}$ by index. Since $F$ is $C^2$-small, one can show that the induced map on associated gradeds is the identity, because the only index zero spiked disks which contribute nontrivially are constants (cf.\ \cite[p.\ 462]{mcduff2012j}). Alternatively, one can extend scalars via $\mathbb{S} \to \mathbb{Z}$. The resulting map is the ordinary PSS map on integral Floer homology, which is well-known to be an isomorphism, and the claimed equivalence then follows by Hurewicz. (The cone of $\Sigma^\infty (M/\partial_\infty M)\to F^0 HF(\overline H^a_\epsilon,\bS)$ is a finite spectrum with vanishing integral homology, which is therefore finite.)

We now turn to the construction of the $S^1$-equivariant PSS map and the proof of \Cref{proposition:equivariant-pss}. The proof essentially consists in rerunning the above arguments equivariantly. To do this, we will use the Morse--Bott techniques developed in  \cite{cotekartalequivariantfloerhomotopy}.

Namely, let $f: M \to \bR$ be Morse as above and let $\tilde{f}: S \to \bR$ be the standard Morse-Bott function introduced in \Cref{subsubsection:data-borel}. Let us now consider the Morse--Bott flow category $\cM_{\tilde f, f}$ where: 
\begin{itemize}
	\item $ob(\cM_{\tilde f, f})=crit(\tilde f)\times crit(f)$
	\item the morphisms are negative gradient trajectories of $\tilde f+f$ with respect to the product metric
\end{itemize}
This is an $S^1$-equivariant Morse--Bott flow category. In fact, $\cM_{\tilde f, f}$ is a trivial instance of the type of flow category considered in \cite[\S4]{cotekartalequivariantfloerhomotopy}. Hence Theorem 4.6 of loc.\ cit.\ implies that the geometric realization $|\cM_{\tilde f, f}|$ is $S^1$-equivariantly equivalent to $\Sigma^\infty (M/\partial_\infty M)$ (strictly speaking we are appealing to an open manifold analog of the aforementioned theorem with the same proof).

To write the promised equivariant PSS map $\Sigma^\infty (M/\partial_\infty M)\to HF_S(H,\bS)$, we shall consider an equivariant $\cM_{\tilde f, f}$-$\cM_{\tilde f, H,J}$-bimodule $\cN_{S}^{PSS}$, similarly to before. To define $\cN_{S}^{PSS}$, the essential idea is to consider pairs given by a negative gradient $\eta$ trajectory of $\tilde f$ and a PSS-trajectory ``above $\eta$''; see \Cref{figure:parametrizedspikeddisc}. To implement this idea rigorously, we choose:
\begin{itemize}
	\item an $S^1$-equivariant function $S\times M\times (-\infty,0]\to\bR$ that agrees with $\tilde f+f$ near $-\infty$ and a family of $S^1$-equivariant metrics on $S\times M$ parametrized by  $(-\infty,0]$ that agrees with the product metric near $-\infty$ 
	\item Floer data on $M$ parametrized by $S\times D$ (where $D=\bC\bP^1\setminus \{0\}$ with fixed negative cylindrical end)  that agrees outside a compact subset of $D$ with the $S\times S^1$-parametric data chosen to define $HF_S(H,\bS)$ (in particular, outside a compact set, the data only depends on the argument of a point in $D$). We further assume that this data is circle equivariant with respect to the diagonal action on $S\times D$.
\end{itemize}
Let $\phi:D\to [0,\infty)$ denote an $S^1$-equivariant smoothing of $q\mapsto |q|^{-1}$ near $q=\infty$, obtained by composing $q\mapsto|q|^{-1}$ with a monotone increasing function $[0,\infty)\to[0,\infty)$ that is equal to $cx^2$ near $0$ and that is equal to the identity outside a small neighborhood of $0$. In other words, $\phi$ is a smoothed projection function $D\to [0,\infty)$. See the right half of \Cref{figure:parametrizedspikeddisc}. 

We define a bimodule $\cN_S^{PSS}$ by assigning to $(p,x)\in ob(\cM_{\tilde f, f})$ and $(p',y)\in (\cM_{\tilde f, H,J})$ the set of $(\eta,\theta,u)$ such that 
\begin{itemize}
	\item $\eta$ is a trajectory of $-\nabla \tilde f$ from $p$ to $p'$
	\item $\theta:(-\infty,0]\to M$ such that $(\eta|_{(-\infty,0]},\theta):(-\infty,0]\to S\times M$  is a negative (half) gradient trajectory with respect to the chosen Morse data above 
	\item $u:D\to M$ satisfies the Cauchy--Riemann equation with respect to the chosen Floer data above. In other words, if we denote the Floer data (temporarily) by $(H_{(s,q)},J_{(s,q)})$ (where $H_{(s,q)}$ denotes the Hamiltonian valued $1$-form on $D$), then we have a Floer datum over $D$ given by $(H_{(\eta(\phi(q)),q)},J_{(\eta(\phi(q)),q)})$ and $u$ satisfies 
	\begin{equation}
		(du-X_{H_{(\eta(\phi(q)),q)}})^{0,1}=0
	\end{equation}
	\item $\theta$ is asymptotic to $x$, $u$ is asymptotic to $y$, and $\theta(0)=u(\infty)$
\end{itemize}

In other words, we are considering pairs of Morse trajectories $\eta$ from $p$ to $p'$ and PSS-trajectories within $S\times M$ from $(p,x)$ to $(p',y)$ ``above $\eta$''; see \Cref{figure:parametrizedspikeddisc}.

By construction, this moduli space admits a \emph{free} $S^1$-action, so transversality is standard (cf.\ \cite{bourgeoisoancea}). The smooth structures are similarly induced from smooth structures on the $S^1$-quotiented moduli spaces, as explained in \cite[Sec.\ 7]{cotekartalequivariantfloerhomotopy}. 
Finally, we have natural framings analogous to \eqref{equation:framing-pss}.
\begin{figure}\centering
	\includegraphics[height=4 cm]{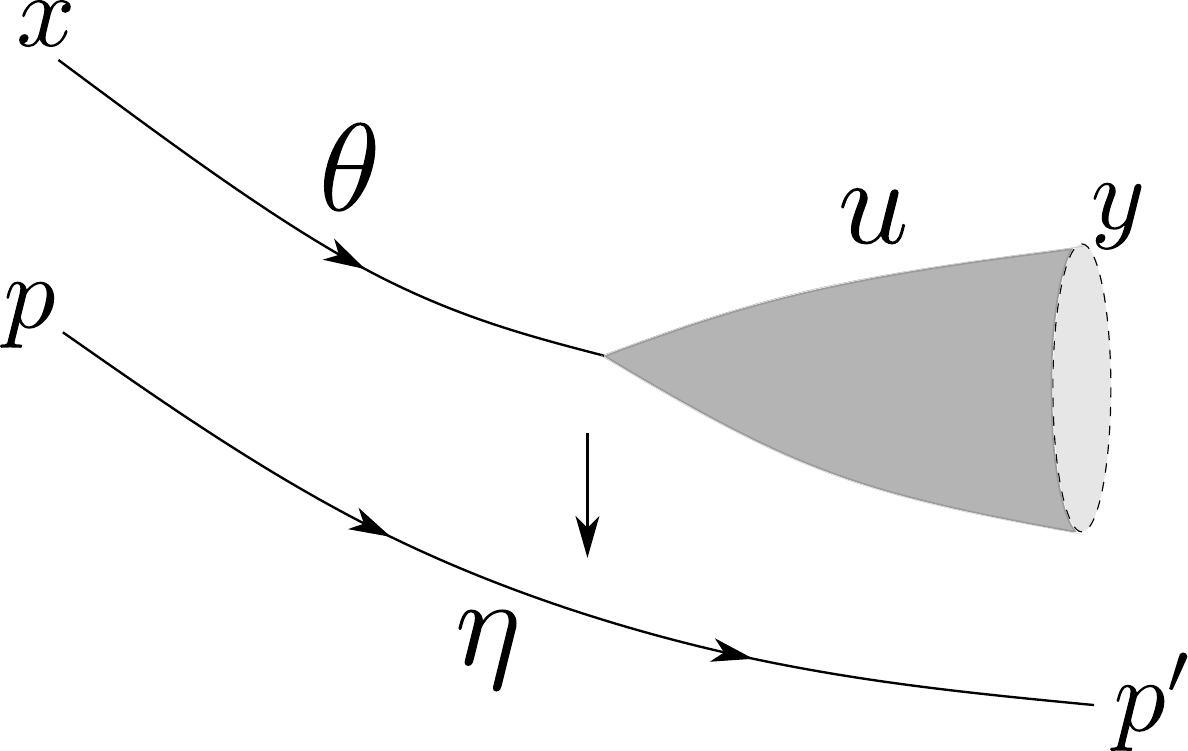}
\caption{A spiked disc ``over $\eta$''}
\label{figure:parametrizedspikeddisc}
\end{figure}

As a result, \cite[Prop.\ 3.4]{cotekartalequivariantfloerhomotopy} furnishes an equivariant PSS-map 
\begin{equation}\label{equation:equivariant-pss}
    \Sigma^\infty (M/\partial_\infty M)\simeq \Sigma^\infty (S\times M/S\times \partial_\infty M) \to HF_S(H,\bS).
\end{equation} 
Similarly to the non-equivariant case, the compatibility with the continuation map follows from standard Floer theory and \cite[\S3.4]{cotekartalequivariantfloerhomotopy}. 

Finally, we prove the equivalence $\Sigma^\infty (M/\partial_\infty M) \xrightarrow{\simeq} F^0HF_S(H^a_\epsilon; \bS)$, corresponding to the second bullet point of \Cref{proposition:equivariant-pss}. 
We will use a filtration on the $S$ component to reduce to the proof of \Cref{lem:nonequivariant-pss} (similarly to \cite[Proof of Thm.\ 4.6]{cotekartalequivariantfloerhomotopy}).

To begin with, and as in the proof of \Cref{lem:nonequivariant-pss}, it is convenient to make a more careful choice of Floer data.  Consider the inclusion $S \times (\mathbb{R} \times S^1)= S \times (D- \{\infty\})  \subset S \times D$. Using the forgetful map $S \times (\mathbb{R} \times S^1) \to \mathbb{R}$, we can think of $S^1$-equivariant Floer data on $S \times (D- \{\infty\})$ as an $\mathbb{R}$-parametrized family of equivariant Floer data over $S \times S^1$.  We write this as $(f_s, H_s)_{s \in \mathbb{R}}$, for $H_s \in Maps(S \times S^1, C^\infty(M))$. For $s  \leq 0$, $(f_s, H_s)_{s \in \mathbb{R}}$ is constant and agrees with the $S \times S^1$-parametric data $(\tilde{f}, H^a_\epsilon)$ constructed in \Cref{subsubsection:borel-jingyu-data}.  We now choose the interpolation of the form $(f_s, \chi(s) H_a^\epsilon + (1-\chi(s)) H_a)$ for $0 \leq s \leq 1$; and on $1 \leq s \leq 2$, we interpolate between $H_a$ and  $e^{-2s} H_a$. We again leave it to the reader to check that this data exends to $S \times D$.
Arguing exactly as in the proof of \Cref{lem:nonequivariant-pss}, this choice ensures that there are no $S$-parametrized PSS trajectories with input a constant trajectory and output a non-constant trajectory. (Note that as $(f_s,H_s)$ varies equivariantly in $S \times S^1$, the induced map is also equivariant.)

The flow categories $\cM_{\tilde f, f}$ and $\cM_{\tilde f, H,J}$ are both filtered by the value of $\tilde f$.  The bimodule $\cN^{PSS}_S$ respects this filtration, so it is enough to prove that the induced map is an equivalence on associated gradeds.  
Let the $S^1$-torsor $X$ be the critical set of $\tilde{f}$ at level $p$. Let $s_0 \in X$. 
The subquotients of $\cM_{\tilde f, H,J}$ do not see the non-constant trajectories, and are homotopy equivalent to the $S^1$-spectrum $X^{T^dX} \wedge |\cM_{\tilde f, H_{s_0},J_{s_0}}|$, where $T^dX$ is the descending tangent bundle of $X$ (i.e.\ the tangent bundle of the descending manifold or equivalently the directions in $TX$ along which the Hessian is negative). 
The subquotients of $\cM_{\tilde f, f}$ are similarly $X^{T^dX} \wedge |\cM_f|$.   
The induced map is just $id \wedge \cN^{PSS}_*$, where $\cN^{PSS}_*$ refers to the non-equivariant PSS map which was already shown to be an equivalence in \Cref{lem:nonequivariant-pss}. This verifies the asserted equivalence on associated gradeds, and thus completes the proof. 

\section{Local Floer homotopy computation}\label{subsec:localfloer}
The goal of this section is to prove the claims \eqref{item:localfloer} and \ref{item:localeq} that were stated in  \Cref{subsubsection:perturbing-autonomous}. We carry over the notation from the previous sections. In particular, we fix a slope $a>0$ and let $\gamma$ be a $k$-fold orbit of the autonomous Hamiltonian $\overline{H}^a$ of action $p$. Let $0<\delta \ll 1$ be as in \Cref{subsubsection:borel-jingyu-data} (in particular, smaller than the action gap between orbits of $\overline{H}^a$). Let $\epsilon$ be as in \Cref{subsubsection:borel-jingyu-data}.\footnote{Recall from \Cref{subsubsection:perturbing-autonomous} that $0<\epsilon \ll \delta$ controls the size of the perturbation of the autonomous Hamiltonian $\overline{H}^a$. When we extend this construction to the Borel-equivariant setting in \Cref{subsubsection:borel-jingyu-data}, then recall that $\epsilon$ is secretly shorthand for a decreasing sequence $\{\epsilon_i\}$, where $\epsilon_i$ controls the size of the perturbation over the unique index $2i$ critical points of $f: \mathbb{CP}^\infty \to \bR$. }

First observe
\begin{lem}\label{lem:onelemma}
A framed flow category $\cM$ with two objects $x,y$, and two morphisms in $\cM(x,y)$ has geometric realization equivalent to either $\Sigma^\infty(S^1)_+$ or $\Sigma^\infty \bR\bP^2$ up to shift. In other words, it is the same as a Thom spectrum over $S^1$.
\end{lem}
\begin{proof}
$\cM$ is the flow category of the standard Morse function on $S^1$, and different framings correspond to the standard framing twisted by a virtual bundle $\nu$ over $S^1$. By \cite[Theorem 4.6']{cotekartalequivariantfloerhomotopy}, $|\cM|\simeq (S^1)^\nu$. The virtual bundles on $S^1$ are are equivalent to the trivial bundle or the M\"obius strip (up to shift). In the former case, $|\cM|$ is equivalent $\Sigma^\infty(S^1)_+$, and in the latter to $\Sigma^\infty \bR\bP^2$ (up to shift). 
\end{proof}

As explained in \Cref{subsubsection:perturbing-autonomous}, it is a consequence of \cite[Lem.\ 2.2]{cieliebakhoferwysockiapplsymplchlg} that the spectrum $F^{(p-\delta,p+\delta]} HF(\overline H^a_\epsilon,\bS)$ is the geometric realization of a flow category satisfying the conditions of \Cref{lem:onelemma}. This implies \eqref{item:localfloer}. However, this argument is harder to implement in the equivariant case. In \cite{cotekartalequivariantfloerhomotopy}, we introduced the formalism of \emph{$P$-relative modules}. Recall that for a space $P$, a $P$-relative module is a module over the flow category with an extra evaluation map into $P$. The examples discussed in loc.\ cit.\ include: (i) For the flow category associated to a Morse/Morse--Bott function, the descending manifolds. The evaluation map is the inclusion into the manifold itself. (ii) For a Liouville manifold $M$, and $P=\cL Q$, the free loop space of an exact Lagrangian $Q\subset M$, the holomorphic half cylinders with boundary on $Q$. The evaluation map is the boundary of the half cylinder which is a loop in $Q$. With an appropriate notion of framing, such modules allow us to define maps into (Thom spectra over) $P$. Here also we will define a map using this formalism, where $P$ is an autonomous orbit. This will give an alternative proof of \eqref{item:localfloer}, which we will then rerun equivariantly to produce the map in \ref{item:localeq}. 

We begin by considering the (non-equivariant) flow category $\cM_{\overline H^a_\epsilon,J}$. 
Let $X_\gamma\simeq S^1/C_k\simeq S^1$ denote the image of $\gamma$. We define a $X_\gamma$-relative module on $F^{p+\delta}\cM_{\overline H^a_\epsilon,J}$ by using the moduli space of continuation trajectories with output given by (some rotation of) $\gamma$. Note that, heuristically, this is like using continuation maps from $HF(\overline H^a_\epsilon,\bS)$ to ``$HF(\overline H^a,\bS)$''. We prefer to avoid formally defining $HF(\overline H^a,\bS)$ as $\overline H^a$ is autonomous and we do not wish to discuss gluing of trajectories along the orbit of an autonomous Hamiltonian.

Choose generic Floer data parametrized by the cylinder that is equal to $(\overline H^a_\epsilon,J)$ on the input (positive) end and whose Hamiltonian term is the same as $\overline H^a+\epsilon$ on the output end. We also assume the Hamiltonian has negative derivative in the $s$-direction. 
Given $x\in ob(F^{p+\delta}\cM_{\overline H^a_\epsilon,J})$, define $\cN(x)$ to be the moduli space of broken trajectories from $x$ to an orbit of $\overline H^a$ supported at $X_\gamma$ (i.e.\ some $\gamma(t+\theta)$). 
\begin{lem}\label{lemma:rel-module-compact}
    $\cN$ is a $X_\gamma$-relative module over $F^{p+\delta}\cM_{\overline H^a_\epsilon,J}$.
\end{lem}
\begin{proof}
 The key point is compactness. Given $x\in ob(F^{p+\delta}\cM_{\overline H^a_\epsilon,J})$, we are asserting that a sequence of trajectories from $x$ to an orbit of $t\mapsto \gamma(t+\theta)$ of $\overline{H}^a$ breaks into a broken Floer trajectory $\overline{H}_\epsilon^a$ from $x$ to some $x' \in ob(F^{p+\delta}\cM_{\overline H^a_\epsilon,J})$, followed by a unbroken trajectory from $x'$ to some $\gamma(t+\theta_0)$. In other words, breakings on the $\overline{H}^a$ side cannot occur.  This is true for elementary action reasons: if our sequence of curves broke on the $\overline{H}^a$ side, then we would have a trajectory from an object of $F^{p+\delta}\cM_{\overline H^a_\epsilon,J}$ to a higher action orbit of $\overline H^a$. This is not possible due to the negative derivative condition in $s$-direction. 

Smoothness of the moduli spaces is achieved by choosing our Floer data generically. Since trajectories can break on the input end, $\cN$ defines a module over $F^{p+\delta}\cM_{\overline H^a_\epsilon,J}$.  Moreover, there is a compatible evaluation map $\cN(x)\to X_\gamma$: send a trajectory that is asymptotic to $\gamma(t+\theta)$ at the negative end to $\gamma(\theta)\in X_\gamma$ (one can identify $X_\gamma$ with these orbits, we are sending a trajectory to the orbit it is asymptotic to). 
\end{proof}
We now discuss framings, which are similar to the standard continuation maps. More precisely, recall that we fix an identification $TM\oplus\bC^k\cong \bC^{n+k}$, and consider the space $\cU_\gamma$ of abstract caps given by the tuples $(y,L,J,Y,g)$ where
\begin{itemize}
	\item $y\in X_\gamma$, also thought of as an orbit of $\overline{H}^a$ of the form $t\mapsto \gamma(t+\theta_0)$
	\item $\epsilon:[0,\infty)\times S^1\to S_-:=\bC$ is a tubular end
	\item $L>0$ is a real number 
	\item $(J,Y)$ is a Floer datum on $S_-\times \bC^{n+k}$ that matches over $\epsilon([L,\infty)\times S^1)$ with the datum for the continuation trajectories under stabilization and the identification $TM\oplus\bC^k\cong \bC^{n+k}$
	\item $g$ is a metric on $S_-$ that matches over $\epsilon([L,\infty)\times S^1)$ with the standard metric
\end{itemize}
This is an autonomous version of \cite[Definition 5.2]{cotekartalequivariantfloerhomotopy}. We have
\begin{enumerate}
	\item a homotopy equivalence $\cU_\gamma\to X_\gamma$ given by $(y,L,J,Y,g)\mapsto y$
	\item a map $\cU_\gamma\to Fred$, the space of Fredholm operators (also see the discussion in \cite[Note 6.7]{cotekartalequivariantfloerhomotopy})
\end{enumerate}

Concretely (1) is just saying that the space of abstract caps over each $y \in X_\gamma$ is contractible; meanwhile (2) is a consequence of the fact that each abstract cap determines a Cauchy--Riemann operator $D_w: W_1^{2, \delta_0}(S_-,\mathbb{C}^{n+k}) \to L^{2, \delta_0}(S_-, \Omega^{0,1}_{S^-} \otimes \mathbb{C}^{n+k})$, which is Fredholm.\footnote{Here $W_1^{2, \delta_0}(\dots), L^{2, \delta}(\dots)$ denotes the subspaces of $W^2_1(\dots), L^2(\dots)$ of sections which decay at infinity at least as fast as $e^{-\delta_0 s}$, for some $0<\delta_0 \ll 1$ which depends only on $M$; see e.g. \cite{bourgeois2002morse, bourgeois2009symplectic} and \Cref{note:weighted-sobolev}.}

We denote by $\nu$ the index bundle over $X_\gamma \simeq U_\gamma$ induced by the map  $\cU_\gamma\to Fred$.  Although we do not need this yet, let us observe already that $\cU_\gamma$ carries a circle action, and both the homotopy equivalence and the map into $Fred$ are compatible with the circle action. As a result (see \cite[Remark 5.7]{cotekartalequivariantfloerhomotopy}) $\nu$ is an $S^1$-equivariant bundle over $\cU_\gamma$ and over $X_\gamma$.

Following \cite[Sec.\ 8]{largethesis}, linear gluing of index bundles furnishes isomorphisms
\begin{equation}
	V_x= T_{\cN(x)}\oplus \nu,
\end{equation}
giving us a framing of $\cN$, in the sense of \cite[Definition 3.5]{cotekartalequivariantfloerhomotopy}.

\begin{note}\label{note:weighted-sobolev}
The appearance of the weighted Sobolev spaces $W^{2,\delta}_1(\dots), L^{2, \delta}(\dots)$ is unavoidable because $X_\gamma$ has a Morse--Bott degeneracy. Since Large's construction of coherent framings \cite[Sec.\ 8]{largethesis} is written in the non-degenerate setting, let us briefly sketch his construction and explain why it generalizes to our setting. Given a curve $u \in \cM(x,y)$ and an abstract cap $w \in \cU(y)$, one first constructs a cap $G(u,w) \in \cU(x)$ by using cutoff functions to glue the underlying Riemann surfaces and Floer data;  given sections $f \in ker(D_{u})$ and $g \in ker(D_w)$, one similarly defines an approximate solution $f \# g$ via cutoffs. This approximate solution does not lie in the kernel of $D_{G(u,w)}$, but it is close (in a precise sense depending on the gluing length). Now compose with the projection to land in the kernel. This defines a map $ker(D_u) \oplus ker(D_w) \to ker(D_{G(u,w)})$; an elaboration of this argument allows one to glue the index bundles (after modifying $D_w, D_{G(u,w)}$ by a linear term to kill the cokernel). To run this construction, one needs to prove estimates which ensure that the approximate solution is close to an honest solution; this is the content of \cite[Lem.\ 8.7]{largethesis}.  Happily, analogous estimates are known in the Morse--Bott setting: indeed, the proof of \cite[Lem.\ 8.7]{largethesis} needs (i) exponential decay estimates and (ii) an estimate on the operator norm of the right inverse.  For (i), see e.g.\ \cite[Prop.\ A.2]{bourgeois2009symplectic} and for (ii), see e.g.\ \cite[Prop.\ 5]{bourgeois2004coherent}.  
\end{note}
\begin{note}
    In contrast to \Cref{note:weighted-sobolev}, generalizing \cite[Sec.\ 6]{largethesis} to Morse--Bott flow categories would require control of \emph{derivatives} of gluing maps, going beyond the estimates available in ``classical'' sources such as \cite{bourgeois2004coherent, bourgeois2009symplectic}. This is precisely the reason why we have not attempted to define $\cN$ as a $\cM_{\overline H^a,J}-\cM_{\overline H^a_\epsilon,J}$-bimodule, and more generally why we never consider flow categories associated to an autonomous Hamiltonian.
\end{note}



Therefore, we get a map 
\begin{equation}\label{eq:thismap}
	F^{p+\delta}HF(\overline H^a_\epsilon,\bS):=|F^{p+\delta}\cM_{\overline H^a_\epsilon,J}|\to (X_\gamma)^\nu.
\end{equation}By the same action considerations as in the proof of \Cref{lemma:rel-module-compact}, given any $x\in ob(F^{p-\delta}\cM_{\overline H^a_\epsilon,J})$, the set $\cN(x)$ is empty. Therefore \eqref{eq:thismap} factors through \begin{equation}\label{eq:localmap}
	F^{(p-\delta,p+\delta]} HF(\overline H^a_\epsilon,\bS)\to  (X_\gamma)^\nu.
\end{equation} 
Before moving on to the equivariant case, we show 
\begin{prop}\label{prop:localeqnoneq}
	\eqref{eq:localmap} is an equivalence.
\end{prop}
Before starting the proof, observe that one can see $(X_\gamma,\nu)$ as a framed Morse--Bott flow category with no morphisms (where $\nu$ is the framing bundle). We denote it by $\cM_\gamma$. Similarly, the module $\cN$ can be seen as a framed $F^{(p-\delta,p+\delta]}\cM_{\overline{H}^a_\epsilon,J}$-$\cM_\gamma$-bimodule. If we had formally constructed the category ``$\cM_{\overline{H}^a,J}$'' associated to the autonomous Hamiltonian $\overline{H}^a$, then $\cM_\gamma$ would be $F^{(p-\delta,p+\delta]}\cM_{\overline{H}^a,J}$. In particular, $\cM_\gamma$ can be interpreted as the local Floer homology of the open set $V_\gamma\subset M$. 

We now introduce a $\cM_\gamma$-$F^{(p-\delta,p+\delta]}\cM_{\overline{H}^a_\epsilon,J}$-bimodule $\cN'$ that will be a partial inverse to $\cN$. Fix Floer data over the cylinder whose Hamiltonian term is $\overline{H}^a$ on the positive (input) end and $\overline{H}^a_\epsilon$ on the negative (output) end. Given $y\in ob(\cM_\gamma)$, $x\in F^{(p-\delta,p+\delta]}\cM_{\overline{H}^a_\epsilon,J}$ define $\cN'(y,x)$ to be the compactified moduli space of cylinders that is asymptotic to $y$ on the positive (input) end (i.e. to some $\gamma(t+\theta)$) and to $x$ on the negative (output) end. This varies continuously in $y$ and defines a $\cM_\gamma$-$F^{(p-\delta,p+\delta]}\cM_{\overline{H}^a_\epsilon,J}$-bimodule. The framing of $\cN'$ is analogous to that of $\cN$.

We would like to warn the reader that $\cN'$ does not extend to a $\cM_\gamma$-$F^{p+\delta}\cM_{\overline{H}^a_\epsilon,J}$-bimodule. Indeed, to extend to such a bimodule, we would need to add trajectories from $y\in X_\gamma$ to lower action orbits of $\overline{H}^a_\epsilon$, as they naturally appear in composition. But the moduli space of such trajectories is not compact, as it does not include the breakings on the autonomous side. Morally, $\cN'$ extends to a ``$F^{p+\delta}\cM_{\overline{H}^a,J}$''-$F^{p+\delta}\cM_{\overline{H}^a_\epsilon,J}$-bimodule; however, we avoid this as we have not defined the former category.  
\begin{proof}[Proof of \Cref{prop:localeqnoneq}]
$\cN'$ gives us a map $X_\gamma^\nu\to 	F^{(p-\delta,p+\delta]}HF(\overline{H}^a_\epsilon,\bS)$. By composing this with $\cN$, we obtain a map $X_\gamma^\nu\to X_\gamma^\nu$. This map is homotopic to a map induced by a $\cM_\gamma$-$\cM_\gamma$-bimodule $\cN''$ that associates to $(y,y')$ the moduli of trajectories from $y$ to $y'$ with respect to continuation data that is equal to $\overline{H}^a$ on the input end, and to $\overline{H}^a+\epsilon$ on the output end. 

This map is homotopic to the identity, hence an isomorphism. More precisely, a $\cM_\gamma$-$\cM_\gamma$-bimodule is simply a smooth manifold with two evaluation maps to $X_\gamma$. By letting $\epsilon$ vary, we see that $\cN''$ is cobordant to a similarly defined bimodule for $\epsilon=0$. But the latter is clearly equal to $X_\gamma$ with evaluation maps given by the identity. The cobordism is compatible with the framings, which tells us the induced maps are homotopic. 

In summary, the composition 
\begin{equation}\label{eq:composelocal}	(X_\gamma)^\nu\to F^{(p-\delta,p+\delta]} HF(\overline H^a_\epsilon,\bS)\xrightarrow{\eqref{eq:localmap}}(X_\gamma)^\nu
\end{equation}
is homotopic to the identity. 

As both sides of \eqref{eq:localmap} are connective, it suffices by the stable Hurewicz theorem to prove the map induced on integral homology is an isomorphism. In other words, if the induced map is an isomorphism, then the cone of \eqref{eq:localmap} is a connective spectrum with vanishing homology; therefore it is zero. 

It follows from \eqref{eq:composelocal} that the composition 
\begin{equation}\label{equation:local-floer-hurewicz}	\widetilde{H}_*(X_\gamma^\nu,\bZ)\to \widetilde{H}_*(F^{(p-\delta,p+\delta]} HF(\overline H^a_\epsilon,\bS),\bZ)\xrightarrow{\eqref{eq:localmap}_*}\widetilde{H}_*(X_\gamma^\nu,\bZ)
\end{equation}
is the identity. In other words, $\widetilde{H}_*(X_\gamma^\nu,\bZ)$ is a direct summand of $\widetilde{H}_*(F^{(p-\delta,p+\delta]} HF(\overline H^a_\epsilon,\bS),\bZ)$. 

Observe that both are isomorphic to a shift of $\bZ\oplus\bZ[1]$ or $\bZ/2\bZ$. Indeed, being a virtual bundle on the circle, $\nu$ is either a trivial bundle, or a M\"obius bundle plus a trivial bundle. In the first case, $(X_\gamma)^\nu$ is a shift of $\Sigma^\infty S^1_+$; in the second case it is a shift of $\Sigma^\infty\bR \bP^2$. The same is true for $F^{(p-\delta,p+\delta]} HF(\overline H^a_\epsilon,\bS)$:  being the realization of a framed flow category $F^{(p-\delta,p+\delta]}\cM_{\overline{H}^a_\epsilon,J}$ with two objects, and two morphisms from one to the other, it is either $\Sigma^\infty S^1_+$ or $\Sigma^\infty\bR \bP^2$, up to a shift by \Cref{lem:onelemma}. This proves the claim about their homology groups. 

The only way for a group isomorphic to $\bZ\oplus\bZ[1]$ or $\bZ/2\bZ$ to be a direct summand of another such group is that they are the same. This implies $\eqref{eq:localmap}_*$ is an isomorphism, finishing the proof.
\end{proof}

Next we turn our attention to the equivariant case. This combines the $X_\gamma$-relative module defined above, with techniques from \cite{cotekartalequivariantfloerhomotopy} which were also used in \Cref{subsec:equivariantPSS}. To conclude the proof, we will use a filtration argument to reduce to \Cref{prop:localeqnoneq}.

We define an $S\times X_\gamma$-relative $S^1$-equivariant framed module $\cN^{eq}$ over $F^{p+\delta}\cM_{\tilde{f}, H^a_\epsilon,J}$. Let $\epsilon:S\to\bR$ be a bounded $S^1$-equivariant function such that $0\leq h\leq \epsilon(s)$ and $\epsilon(b_i)=\epsilon_i$. Choose Floer data parametrized by $S\times \bR\times S^1$ that is 
\begin{itemize}
	\item equivariant with respect to the diagonal $S^1$-action on $S\times S^1$
	\item the Hamiltonian term is non-increasing in the $\bR$ direction 
	\item the Hamiltonian term is equal to $H^a_\epsilon=H+h:S\times S^1\times M\to\bR$ on the input (positive) end
	\item the Hamiltonian term is equal to $H^a+\epsilon(s):S\times M\to\bR$ on the output (negative) end
\end{itemize}

To define $\cN^{eq}$, we associate to a pair $(q,x)\in ob (F^{p+\delta}\cM_{\tilde{f}, H^a_\epsilon,J})$ the moduli of broken pairs $(\beta,u)$ where
\begin{itemize}
	\item $\beta:(-\infty,0]\to\bR$ is a negative, half gradient trajectory of $\tilde{f}$ from $q$
	\item $u:\bR\times S^1\to M$ is a solution from $x$ to some $t\mapsto\gamma(t+\theta)$ to the equation \begin{equation}\label{equation:paramterized-continuation}
		\partial_s u+J_{(\beta(-e^s),s,t)}(\partial_t u- X_{H_{(\beta(-e^s),s,t)}})=0
	\end{equation}
\end{itemize}
In other words, $u$ is a continuation trajectory from $x$ to $\gamma$ ``over $\beta(-e^s):\bR\to S$''. 
\begin{lem}
$\cN^{eq}$ is a module over $F^{p+\delta}\cM_{\tilde{f}, H^a_\epsilon,J}$. 
\end{lem}
\begin{proof}
First of all, observe that for $s \gg 1$, $\beta(s)$ is contained in a slice of the $S^1$-action and $(H_{(\beta(-e^s),s, t)}, J_{(\beta(-e^s),s, t)})$ are independent of $s$ on this slice for $s \gg 1$. This follows immediately from the construction of $H^a_\epsilon, J$ and of the $S^1$-equivariant metric on $S$, all of which is explained in \Cref{subsubsection:data-borel} and \Cref{subsubsection:borel-jingyu-data}. Hence \eqref{equation:paramterized-continuation} reduces to the standard Cauchy--Riemann equation for $s \gg 1$. 

For $s \ll -1$, the Hamiltonian term of the pair $(H_{(\beta(-e^s),s, t)}, J_{(\beta(-e^s),s, t)})$ converges exponentially in $s$ to $H^a_{(\beta(0), \epsilon(\beta(0)),t)}=\overline{H}^a + \epsilon(\beta(0))$, and similarly with its $J$-term.
%
Similarly to the proof of \Cref{lemma:rel-module-compact}, we observe that there can be no breaking as $s \to -\infty$ (output). Indeed, any such breaking would cause the appearance of a trajectory from an orbit of $H^a_\epsilon(b,\cdot,\cdot)$ of action less than $p+\delta$ to an orbit of $\overline{H}^a + \epsilon(\beta(0))$ of higher action. This is impossible due to our assumption that the Hamiltonian term has non-increasing derivative in the $\bR$-direction. 

We do not spell out the details of transversality and gluing, which follow from similar considerations as for $\cN_S^{PSS}$ in \Cref{subsec:equivariantPSS}. 
\end{proof}
We define the evaluation map  by 
\begin{equation}
	\cN^{eq}(q,x)\to S\times X_\gamma\atop (\beta,u)\mapsto (\beta(0),\gamma(\theta))
\end{equation}
where $\gamma(t+\theta)$ is the output asymptotic of $u$. It carries an equivariant framing relative to the pull-back of the virtual equivariant bundle $\nu$ to $S\times X_\gamma$, which we still denote by $\nu$. As a result, we have an $S^1$-equivariant map
\begin{equation}\label{eq:thisequivmap}
	F^{p+\delta}HF_S(H^a_\epsilon,\bS):=|F^{p+\delta}\cM_{\tilde f, H^a_\epsilon,J}|\to \Sigma^\infty S_+\wedge (X_\gamma)^\nu\simeq (X_\gamma)^\nu.
\end{equation} As before, $\cN^{eq}$ is empty over $ob(F^{p-\delta}\cM_{\tilde f,\overline H^a_\epsilon,J})$; therefore, \eqref{eq:thisequivmap} factors through \begin{equation}\label{eq:localequivmap}
	F^{(p-\delta,p+\delta]} HF_S(\overline H^a_\epsilon,\bS)\to  \Sigma^\infty S_+\wedge (X_\gamma)^\nu\simeq (X_\gamma)^\nu.
\end{equation} 
\begin{prop}\label{prop:localequiv}
\eqref{eq:localequivmap} is an equivalence. 
\end{prop}
To prove this, we reduce to \Cref{prop:localeqnoneq} by using a filtration argument.
\begin{proof}
$F^{(p-\delta,p+\delta]} HF_S(\overline H^a_\epsilon,\bS)$ is given by the geometric realization of $\cM_{loc}:=F^{(p-\delta,p+\delta]} \cM_{\tilde{f}, H^a_\epsilon,J}$, which carries another filtration by the value of $\tilde{f}$. The category $F^r\cM_{loc}=F^rF^{(p-\delta,p+\delta]} \cM_{\tilde{f}, H^a_\epsilon,J}$ jumps when $r_i=\tilde{f}(b_i)$ for some $i$, i.e.\ at the critical values of $\tilde{f}$. Let $X_i$ denote the $S^1$-orbit of $b_i$, i.e.\ the critical set of $\tilde{f}$ of index $2i$. Observe that the quotient category $F^{r_i}\cM_{loc}/F^{r_{i-1}}\cM_{loc}$ can be identified with the Morse--Bott flow category given by taking the product of $X_i$ with the ordinary flow category $F^{(p-\delta,p+\delta]}\cM_{\overline{H}^a_{\epsilon_i},J}$. Here $F^{(p-\delta,p+\delta]}\cM_{\overline{H}^a_{\epsilon_i},J}$ is the flow category corresponding to restriction of the $S$-parametric Floer data to $b_i$. In other words, we take products of every object and the morphism space with $X_i$. The framing bundles on $X_i\times F^{(p-\delta,p+\delta]}\cM_{\overline{H}^a_{\epsilon_i},J}$ are the sum of these for $F^{(p-\delta,p+\delta]}\cM_{\overline{H}^a_{\epsilon_i},J}$ plus the descending bundle $V_i$ over $X_i$. Its realization is the same as $X_i^{V_i}\wedge F^{(p-\delta,p+\delta]}HF(\overline{H}^a_\epsilon,\bS)$. 

Similarly, the target $S\times X_\gamma$ is filtered by the value of $\tilde{f}$, and jumps in the homotopy type only occur at the critical points $r_i$. Indeed, this filtration depends only on the $S$-component, and $F^{r_i}\Sigma^\infty S_+/F^{r_{i-1}}\Sigma^\infty S_+$ is equivalent to $X_i^{V_i}$ (see \cite[Sec.\ 4]{cotekartalequivariantfloerhomotopy} for more details). Therefore we have a map
\begin{equation}\label{eq:somelongmap}
X_i^{V_i}\wedge F^{(p-\delta,p+\delta]}HF(\overline{H}^a_\epsilon,\bS) \simeq |F^{r_i}\cM_{loc}|/|F^{r_{i-1}}\cM_{loc}|\to\atop 	F^{r_i}\big(\Sigma^\infty S_+\wedge (X_\gamma)^\nu\big)\big/ 	F^{r_{i-1}}\big(\Sigma^\infty S_+\wedge (X_\gamma)^\nu\big)\simeq X_i^{V_i}\wedge X_\gamma^\nu
\end{equation}
and it is not difficult to see that this map matches \eqref{eq:localmap} smashed with $X_i^{V_i}$. Hence, by \Cref{prop:localeqnoneq}, the map \eqref{eq:somelongmap} is an equivalence for every $i$. As a result, by induction, the map
\begin{equation}
	F^{r_i}HF_S(H^a_\epsilon)\simeq |F^{r_i}\cM_{loc}|\to F^{r_i} \big(\Sigma^\infty S_+\wedge (X_\gamma)^\nu\big)
\end{equation}
induced by $\cN^{eq}$ is an isomorphism for every $r_i$. This implies the same for their colimit as $i\to\infty$, which concludes the proof. 
\end{proof}

\bibliographystyle{alpha}
\bibliography{biblio}

\begin{thebibliography}{CFHW96}

\bibitem[Abo11]{abouzaid2011cotangent}
Mohammed Abouzaid.
\newblock A cotangent fibre generates the fukaya category.
\newblock {\em Advances in Mathematics}, 228(2):894--939, 2011.

\bibitem[Abo15]{abouzaid2013symplectic}
Mohammed Abouzaid.
\newblock Symplectic cohomology and {V}iterbo's theorem.
\newblock In {\em Free loop spaces in geometry and topology}, volume~24 of {\em
  IRMA Lect. Math. Theor. Phys.}, pages 271--485. Eur. Math. Soc., Z\"{u}rich,
  2015.

\bibitem[ACF16]{albers2016symplectic}
Peter Albers, Kai Cieliebak, and Urs Frauenfelder.
\newblock Symplectic tate homology.
\newblock {\em Proceedings of the London Mathematical Society},
  112(1):169--205, 2016.

\bibitem[AD14]{audindamianmorse}
Mich\`ele Audin and Mihai Damian.
\newblock {\em Morse theory and {F}loer homology}.
\newblock Universitext. Springer, London; EDP Sciences, Les Ulis, 2014.
\newblock Translated from the 2010 French original by Reinie Ern\'{e}.

\bibitem[AMS21]{abouzaidmcleansmithglobal1}
Mohammed {Abouzaid}, Mark {McLean}, and Ivan {Smith}.
\newblock {Complex cobordism, Hamiltonian loops and global Kuranishi charts}.
\newblock {\em arXiv e-prints}, page arXiv:2110.14320, October 2021.

\bibitem[BB20]{barthelbeaudrychromatic}
Tobias Barthel and Agn\`es Beaudry.
\newblock Chromatic structures in stable homotopy theory.
\newblock In {\em Handbook of homotopy theory}, CRC Press/Chapman Hall Handb.
  Math. Ser., pages 163--220. CRC Press, Boca Raton, FL, [2020] \copyright
  2020.

\bibitem[BM04]{bourgeois2004coherent}
Fr{\'e}d{\'e}ric Bourgeois and Klaus Mohnke.
\newblock Coherent orientations in symplectic field theory.
\newblock {\em Mathematische Zeitschrift}, 248:123--146, 2004.

\bibitem[BO09]{bourgeois2009symplectic}
Fr{\'e}d{\'e}ric Bourgeois and Alexandru Oancea.
\newblock Symplectic homology, autonomous hamiltonians, and morse-bott moduli
  spaces.
\newblock {\em Duke mathematical journal}, 146(1):71--174, 2009.

\bibitem[BO17]{bourgeoisoancea}
Fr\'{e}d\'{e}ric Bourgeois and Alexandru Oancea.
\newblock {$S^1$}-equivariant symplectic homology and linearized contact
  homology.
\newblock {\em Int. Math. Res. Not. IMRN}, (13):3849--3937, 2017.

\bibitem[Bou02]{bourgeois2002morse}
Fr{\'e}d{\'e}ric Bourgeois.
\newblock {\em A Morse-Bott approach to contact homology (PhD Thesis)}.
\newblock PhD thesis, stanford university, 2002.

\bibitem[CFHW96]{cieliebakhoferwysockiapplsymplchlg}
K.~Cieliebak, A.~Floer, H.~Hofer, and K.~Wysocki.
\newblock Applications of symplectic homology. {II}. {S}tability of the action
  spectrum.
\newblock {\em Math. Z.}, 223(1):27--45, 1996.

\bibitem[CHHL23]{dancgtwoinfty}
Dan {Cristofaro-Gardiner}, Umberto {Hryniewicz}, Michael {Hutchings}, and Hui
  {Liu}.
\newblock {Proof of Hofer-Wysocki-Zehnder's two or infinity conjecture}.
\newblock {\em arXiv e-prints}, page arXiv:2310.07636, October 2023.

\bibitem[Cie02]{cieliebak2002handle}
Kai Cieliebak.
\newblock Handle attaching in symplectic homology and the chord conjecture.
\newblock {\em Journal of the European Mathematical Society}, 4:115--142, 2002.

\bibitem[CK23]{cotekartalequivariantfloerhomotopy}
Laurent C\^{o}t\'e and Yusuf~Bar{\i}\c{s} Kartal.
\newblock {Equivariant Floer homotopy via Morse-Bott theory}.
\newblock {\em arXiv e-prints}, page arXiv:2309.15089, 2023.

\bibitem[FOOO16]{foooexponential}
Kenji Fukaya, Yong-Geun Oh, Hiroshi Ohta, and Kaoru Ono.
\newblock Exponential decay estimates and smoothness of the moduli space of
  pseudoholomorphic curves, 2016.

\bibitem[Gan12]{sheelthesis}
Sheel Ganatra.
\newblock {\em Symplectic {C}ohomology and {D}uality for the {W}rapped {F}ukaya
  {C}ategory}.
\newblock ProQuest LLC, Ann Arbor, MI, 2012.
\newblock Thesis (Ph.D.)--Massachusetts Institute of Technology.

\bibitem[{Gan}19]{ganatracyclic}
Sheel {Ganatra}.
\newblock {Cyclic homology, $S^1$-equivariant Floer cohomology, and Calabi-Yau
  structures}.
\newblock {\em arXiv e-prints}, page arXiv:1912.13510, December 2019.

\bibitem[H{\'a}j14]{hajek2014manifolds}
Pavel H{\'a}jek.
\newblock On manifolds with corners, 2014.

\bibitem[HHLN23]{haugseng2023lax}
Rune Haugseng, Fabian Hebestreit, Sil Linskens, and Joost Nuiten.
\newblock Lax monoidal adjunctions, two-variable fibrations and the calculus of
  mates.
\newblock {\em Proceedings of the London Mathematical Society},
  127(4):889--957, 2023.

\bibitem[Lar21]{largethesis}
Tim Large.
\newblock Spectral fukaya categories for liouville manifolds.
\newblock 2021.

\bibitem[MS12]{mcduff2012j}
Dusa McDuff and Dietmar Salamon.
\newblock {\em J-holomorphic curves and symplectic topology}, volume~52.
\newblock American Mathematical Soc., 2012.

\bibitem[PS24]{porcelli2024bordism}
Noah Porcelli and Ivan Smith.
\newblock Bordism of flow modules and exact lagrangians.
\newblock {\em arXiv:2401.11766}, 2024.

\bibitem[Seg68]{segalequivariantktheory}
Graeme Segal.
\newblock Equivariant {$K$}-theory.
\newblock {\em Inst. Hautes \'{E}tudes Sci. Publ. Math.}, (34):129--151, 1968.

\bibitem[Sei08]{seidel2008biased}
Paul Seidel.
\newblock A biased view of symplectic cohomology.
\newblock In {\em Current developments in mathematics, 2006}, volume 2006,
  pages 211--254. International Press of Boston, 2008.

\bibitem[Tor23]{torii2023perfect}
Takeshi Torii.
\newblock A perfect pairing for monoidal adjunctions.
\newblock {\em Proceedings of the American Mathematical Society},
  151(12):5069--5080, 2023.

\bibitem[Tre19]{treumann2019complex}
David Treumann.
\newblock Complex k-theory of mirror pairs.
\newblock {\em arXiv preprint arXiv:1909.03018}, 2019.

\bibitem[Weh12]{wehrheim2012smooth}
Katrin Wehrheim.
\newblock Smooth structures on morse trajectory spaces, featuring finite ends
  and associative gluing.
\newblock {\em Proceedings of the Freedman Fest}, 18:369--450, 2012.

\bibitem[Wur91]{wurglermorava}
Urs Wurgler.
\newblock Morava {$K$}-theories: a survey.
\newblock In {\em Algebraic topology {P}ozna\'{n} 1989}, volume 1474 of {\em
  Lecture Notes in Math.}, pages 111--138. Springer, Berlin, 1991.

\bibitem[Zha19]{zhaoperiodic}
Jingyu Zhao.
\newblock Periodic symplectic cohomologies.
\newblock {\em Journal of Symplectic Geometry}, 17(5):1513--1578, 2019.

\end{thebibliography}

\end{document}